\documentclass[reqno, 11pt]{amsart}
\usepackage[bbgreekl]{mathbbol}
\usepackage{amsmath,amssymb,amsthm,amsfonts, amscd, url}
\usepackage{appendix}
\usepackage[numeric]{amsrefs}
\usepackage{hyperref}
\input xy
\xyoption{all}

\usepackage[dvipsnames]{xcolor}
\usepackage{graphics}
\usepackage{graphicx}

\newcommand{\red}[1]{{\color{red}#1}}

\usepackage{mathptmx} 
\usepackage[scaled=0.90]{helvet} 
\usepackage{courier} 
\normalfont
\usepackage[T1]{fontenc}
\newcommand{\bmu}{\bbmu}

\newcommand{\z}{\mathbf{z}}
\DeclareMathOperator{\gen}{gen}
\DeclareMathOperator{\Av}{Av}

\DeclareMathOperator{\iw}{Iw}

\newcommand{\ve}{\mathbf{v}}

\newcommand{\bv}{b^{\vee}}

\newtheorem*{nthm}{Theorem}
\newtheorem*{nprop}{Proposition}
\newtheorem*{nlem}{Lemma}
\newtheorem*{ncor}{Corollary}
\newtheorem*{nclaim}{Claim}
\newtheorem*{nrem}{Remark} 

\theoremstyle{remark}
\newtheorem*{nrems}{Remark}

\newcommand{\be}[1]{\begin{eqnarray} \label{#1}}

\newcommand{\ee}{\end{eqnarray}}

\newcommand{\tpoint}[1]{\vspace{2mm}\par \noindent \refstepcounter{subsection}{\bf \thesubsection.}
  {\em #1.} }

\newcommand{\spoint}{\vspace{2mm}\par \noindent \refstepcounter{subsection}{\bf \thesubsection.} }

\numberwithin{equation}{section}

\newcommand{\sph}{\mathbf{1}_K}

\newcommand{\xv}{\xi^{\vee}}

\newcommand{\gv}{\gamma^{\vee}}

\newcommand{\av}{a^{\vee}}

\newcommand{\aw}{\mathbb{W}}

\newcommand{\up}{U}
\newcommand{\um}{U^-}

\newcommand{\valu}{\mathsf{val}}

\newcommand{\la}{\langle}
\newcommand{\ra}{\rangle}
\newcommand{\lv}{\Lambda^{\vee}}
\newcommand{\mv}{\mu^{\vee}}

\newcommand{\mf}[1]{\mathfrak{#1}}
\newcommand{\f}[1]{\mathfrak{#1}}
\newcommand{\rr}{\rightarrow}
\newcommand{\mc}[1]{\mathcal{#1}}

\newcommand{\Ps}{\theta}

\newcommand{\Lv}{\Lambda^{\vee}}
\renewcommand{\lv}{\lambda^{\vee}}

\newcommand{\tauv}{\tau^{\vee}}

\newcommand{\zee} {\mathbb{Z}}
\newcommand{\C} {\mathbb{C}}

\renewcommand{\O}{\mathcal{O}}

\newcommand{\niw}{\textbf{n}^-}

\renewcommand{\c}{\mathbf{c}}
\renewcommand{\b}{\mathbf{b}}
\newcommand{\B}{\mathsf{B}}

\newcommand{\W}{\mathcal{W}}

\newcommand{\resi}{\mathsf{res}}

\newcommand{\ZZ}{\mathbb{Z}}

\newcommand{\kk}{\kappa}

\begin{document}

\title{On Iwahori-Whittaker Functions for Metaplectic Groups}
\author{Manish Patnaik \&  Anna  Pusk\'{a}s}
\date{\today}
\address{University of Alberta \\ Department of Mathematical and Statistical Sciences, CAB 632 \\ Edmonton, Alberta T6G 2G1}
\email{patnaik@ualberta.ca, puskas@ualberta.ca}

\newcommand{\tA}{\widetilde{A}}
\newcommand{\Pshi}{\widetilde{\Phi}}
\newcommand{\tG}{\widetilde{G}}
\newcommand{\tK}{\widetilde{K}}
\newcommand{\tU}{\widetilde{U}}
\newcommand{\an}{\mathbf{a}}
\newcommand{\nn}{\mathbf{n}^-}
\newcommand{\tT}{\widetilde{\mathbf{T}}}
\newcommand{\tB}{\widetilde{B}}
\newcommand{\tGg}{\tG^{gen}}

\newcommand{\un}{\mf{u}}
\newcommand{\mq}{\mathbf{\mu}_{q-1}}
\newcommand{\Q}{\mathsf{Q}}
\newcommand{\m}{\mathsf{n}}
\newcommand{\tH}{\widetilde{H}}

\newcommand{\g}{\mathbf{g}}

\newcommand{\T}{\mathbf{T}}
\newcommand{\spw}{\widetilde{\mathbf{1}}_{\psi}}
\newcommand{\hs}{\widetilde{\mathbf{1}}}

\newcommand{\tve}{\widetilde{\ve}}

\newcommand{\sx}{\mathsf{x}}
\newcommand{\sh}{\mathsf{h}}
\newcommand{\sw}{\mathsf{w}}
\newcommand{\bh}{\mathbf{h}}
\newcommand{\bx}{\mathbf{x}}
\newcommand{\bw}{\mathbf{w}}
\newcommand{\bE}{\mathcal{E}}
\newcommand{\bA}{\mathbf{A}}
\renewcommand{\sc}{\mathsf{c}}

\newcommand{\tI}{\widetilde{I}}
\newcommand{\tIm}{\widetilde{I}^-}

\newcommand{\mult}{S}
\newcommand{\ring}{\C_{v,\mult}[\Lv]}
\newcommand{\nice}{\C_{v,\mult}[\Lv _0]}
\newcommand{\nicer}{\C_{v}(\widetilde{\Lambda}^{\vee})}
\newcommand{\nuv}{{\nu^{\vee}}}
\newcommand{\nicetwist}{\nicer [\la s, t \ra_\star]}

\maketitle

\begin{abstract} We relate Iwahori-Whittaker functions on metaplectic covers to certain Demazure-Lusztig operators, the latter of which are built from a Weyl group action previously considered by G. Chinta and P. Gunnells. Using a certain combinatorial identity for the sum of these Demazure-Lusztig operators, we obtain an analogue of the Casselman-Shalika formula for spherical Whittaker functions in this context.  \end{abstract}

\tableofcontents

\section{Introduction}

\spoint For a positive integer $n$, let $\tG$ be the $n$-fold metaplectic cover of a split, simple, simply-laced group $G$ over a non-archimedean local field. The aim of this note is to make a direct link between (unramified) Whittaker functions on $\tG$ and certain metaplectic Demazure-Lusztig operators that were recently introduced by the second named author together with G. Chinta and P. Gunnells \cite{cgp}. Namely, the metaplectic Whittaker function can be decomposed into pieces, the so-called Iwahori-Whittaker functions, and our main result (Theorem \ref{main-rec}) is that each of these pieces coincides with the application of a metaplectic Demazure-Lusztig operator. As the entire Whittaker function is a sum of Iwahori-Whittaker functions, it can also be expressed as a sum of Demazure-Lusztig operators.  A combinatorial identity for the same sum of operators was obtained in \cite[Theorem 4]{cgp}; combining this identity with our main result yields the Casselman-Shalika formula in this metaplectic setting. Using different techniques, this formula was obtained earlier by G. Chinta and O. Offen in \cite{ch:of} for $\mathrm{GL}_r,$ and for more general groups by P. McNamara in \cite{mac-met}. We also compute the metaplectic spherical function using similar ideas (see Appendix \ref{app-sph}). As far as we know, this formula has not appeared in the literature except in the case of $GL_r$ (see \cite{ch:of}).

\spoint What appears to be new here is the connection between a general metaplectic Iwahori-Whittaker functions and metaplectic Demazure-Lusztig operators (see Theorem \ref{main-rec} or \eqref{W:T-1}). Both these sets of objects are indexed by the Weyl group, and previously this connection was only known for the long element in the Weyl group and that too for a somewhat indirect reason. Note that in the finite dimensional setting (i.e. for root systems of finite type) just knowing this connection for the long word is essentially equivalent to the Casselman-Shalika formula. On the other hand, for a general Kac-Moody group, where there is no longest element in the Weyl group, a connection between arbitrary Iwahori-Whittaker functions and Demazure-Lusztig operators is a necessary ingredient in our strategy to derive the  Casselman-Shalika formula.   The techniques employed here to establish our main Theorem \ref{main-rec} are similar to those of \cite{pat-whit}, and this Theorem can be generalized to metaplectic covers of Kac-Moody groups (note that these groups have not yet appeared in the literature).  However, to obtain the final Casselman-Shalika formula in this new context (i.e., the analogue of Corollary \ref{main-cs-fin})) further complications entirely of a Kac-Moody nature need to be addressed. It was our aim in writing this note to first deal with the purely metaplectic issues, leaving the the Kac-Moody ones to a future work.

\spoint We make a few comments about the techniques of this paper. As we said above, the main recursion result is derived using a strategy analogous to the one in \cite{pat-whit}, which was written in the non-metaplectic (though affine) setting. The essential differences that arise in adapting this strategy to the metaplectic context can already be perceived in the rank one situation, and so we carry out this computation for $\mathrm{SL}_2$ in some detail in \S \ref{rk1case}. 

In the non-metaplectic case, checking the braid relations for the Demazure-Lusztig operators is but a simple computation (see \cite[(8.3)]{lus-K}). In the metaplectic case, the situation is complicated by the involved nature of the Chinta-Gunnells Weyl group action out of which these operators are built. The original proofs of these braid relations were computational in nature-- they proceeded by first checking algebraically (see \cite{ch:gu}) that the Chinta-Gunnells formulas define a Weyl group action, and then verifying the braid relations for the Demazure-Lusztig operators via a further computation (see \cite{cgp}). From our work, the braid relations for both the Chinta-Gunnells formulas as well as the Demazure-Lusztig operators essentially follow. For the Demazure-Lusztig operators, using an idea very similar to one used in \cite{bbl}, the braid relations (at least for the application to dominant coweights) follow from our main recursion result as we explain in Remark \ref{rmk:DL-braid}. For the braid relations satisfied by the Chinta-Gunnells formulas, we actually offer two proofs: the first approach, explained in \S \ref{braid-cg} and similar in spirit to the argument in McNamara \cite{mac-met}, makes use of an "intertwiner" interpretation (see \eqref{Ia:cg}) of the Chinta-Gunnells formulas. The  second approach is explained in Appendix \ref{app:braid} and gives a bit more, namely the braid relations for Demazure-Lusztig operators applied to all coweights (not just dominant ones). This proof uses some facts about rank two root systems, a reduction to the non-metaplectic setting, as well as our main recursion result .

In appendix \ref{app-sph}, we compute the metaplectic spherical function following \cite[\S7]{bkp} closely.  The computation is carried out using the same stragegy as in the Whittaker case, i.e. reducing to the Iwahori-case, which can then be expressed using some Demazure-Lusztig type operators. Actually, the arguments are simpler as the relevant Demazure-Lusztig operators in the spherical case are built from the ordinary Weyl group action rather than the Chinta-Gunnells action. Aside from illustrating the similarity to the Whittaker case, our interest in the spherical case is due to the fact that in the metaplectic, \emph{affine} setting the only method known to us which obtains the full (i.e., not just up to a Weyl group invariant constant) Casselman-Shalika formula goes through the spherical function (see Remark \ref{whit:spy} for more on this point).

\tpoint{Relation to previous work} The connection between Iwahori-Whittaker functions and Demazure-Lusztig operators was first noticed in the nonmetaplectic setting by B. Brubaker, D. Bump, and A. Licata \cite{bbl}. Our Theorem \ref{main-rec} is a metaplectic analogue of \cite[Theorem 1]{bbl}, though we note that our technique, if adapted to non-metaplectic setting, is actually not the same as the one in \emph{loc. cit} as we do not invoke here any uniqueness principle for Whittaker functionals. 

In this paper, we define the Whittaker functions in terms of certain explicit integrals and essentially avoid the usual representation theoretical framework for dealing with such functions. We do this partly with an eye toward the Kac-Moody setting where, although such representation theoretical constructions are not yet defined, the constructions of this paper can still be carried out. To connect however with the usual notions (i.e. Whittaker functionals on principal series representations), we make some remarks in \S \ref{s:whit-fun}. Also, in \S \ref{s:kaz-pat} we include some remarks which connect our main recursion result (and hence the Chinta-Gunnells action) with the earlier work of Kazhdan and Patterson \cite{ka:pat} on the composition of Whittaker functionals and intertwining operators. This connection is certainly not new, and a similar type of relation to the work of \cite{ka:pat} was one of the sources of motivation for the introduction by Chinta and Gunnells \cite[p.6]{ch:gu} for their novel Weyl group action.  

In the literature, there exists a different formula for metaplectic Whittaker functions which is written in terms of Kashiwara's crystal basis. This formula was first obtained by B. Brubaker, D.Bump, and S. Friedberg for metaplectic covers of $\mathrm{SL}_r$ (see \cite{bbf-annals} and references therein); a connection to crystals for more general groups was made by McNamara using a different, geometric approach in \cite{mac-cryst}.  In case of $\mathrm{SL}_r$, specializing the formulas obtained from this crystal theoretic approach to the case of $n=1$ (i.e., non-metaplectic setting), one obtains a formula for the Whittaker function different from the usual Casselman-Shalika formula, the two being related by a (non-trivial) combinatorial identity of Tokuyama (see \cite{tok}).   In the case of metaplectic ${\mathrm{SL}_r},$ a direct connection (one can also derive this from the work of McNamara \cite{mac-cryst}) between the crystal theoretic formula and the Casselman-Shalika formula of \cite{ch:of} (our Corollary \ref{main-cs-fin}) was made by the second named author in \cite{pus-th} using the combinatorics of the metaplectic Demazure-Lusztig operators. More precisely, an identity is phrased in \emph{loc. cit} that relates the formulas of \cite{bbf-annals} to those of \cite{ch:of}, generalizing the work of Tokuyama.

\tpoint{On our assumptions}\label{tpoint:assumptions} In this paper, we make some rather restrictive assumptions. We are not aiming to show the maximal generality in which our methods work, but rather to illustrate in a representative class how our techniques apply. So we assume that the underlying non-metaplectic group is split, simple, and simply connected. The reason is so that we can directly use the construction of Steinberg and Matsumoto \cite{mat}, which is written in this generality. If one wants to consider a general \emph{split} reductive group, one can work with the covers constructed by Deligne-Brylinski \cite{del-bry}, and we expect then that our construction will carry over with only minor modifications. To remove the splitness assumption requires a little more care, but some of the necessary rank one computations are already worked out in McNamara in \cite{mac-met} for unramified groups. It seems that combining these calculations and a general descent process also described in \emph{loc. cit}, we can extend the present work to unramified groups. 

A more intricate set of restrictions to lift are those imposed between the local field $\mc{K}$ and the degree of the metaplectic cover $n.$  Specifically we assume that we are in the tame case, i.e. $(q, n)=1$ and also that $q \equiv 1 \mod 2n.$ The former assumption is necessary for us to assert that the maximal compact of $G$ has a splitting in the metaplectic cover.  For a general (reductive, not-necessary split) group, the splitting in covers of maximal compact subgroups has been studied by Gan and Gao \cite[Cor. 4.2]{gg} and the situation seems subtle outside the simply connected case. Finally we comment on the assumption that $q \equiv 1 \mod 2n:$ while practically allowing us to avoid certain sign issues (see for example \S \ref{s:hilb-sym}), it also hides the full complexity of the true metaplectic $L$-group which Weismann has constructed (see \cite[\S 4.4.3]{weis}). It would be very interesting to see how to remove this assumption.

%

\tpoint{Acknowledgements} M.P. and A.P. were supported from the Subbarao Professorship in Analytic Number Theory, an NSERC discovery grant, and an University of Alberta startup grant.  

M.P. would like to thank Alexander Braverman for some useful discussions during the Park City Math Institute pertaining to the material in  \S \ref{braid-cg}. 

We would also like to thank the anonymous referee for several thoughtful suggestions.

\section{Basic Notations on Local Fields} 

\spoint Let $\mc{K}$ be a non-archimedean local field with ring of integers $\O$ and valuation map $\valu: \mc{K}^* \rr \zee.$ Let $\pi \in \O$ be a uniformizing element, and $\kk:= \O / \pi \O$ be the residue field, whose size we denote by $q.$ The norm associated to $\valu$ will be $|x|:= q^{ - \valu(x) } $ for $x \neq 0$ and $| 0 | = 0.$ Denote by $\varpi: \O \rr \kappa$ the natural quotient map.

\spoint Let $du$ be the Haar measure on $\mc{K}$ giving $\O$ volume $1.$ For $k \geq 0$ we denote the set $\O^*[k]= \{ x \in \mc{K}^* \mid \valu(x) = k \}$ (so $\O^*=\O^*[0]$) and the group $\O(k) = \{ x \in \mc{K}^* \mid \valu(x) \geq k \}$ (so $\O=\O(0)$) We have $\O^*[k]= \O(k) \setminus \O(k+1).$ The fixed Haar measure gives $\O(k)$ volume $q^{-k}$ and $\O^*[k]$ volume $q^{-k}(1-q^{-1}).$

\spoint \label{psi-sec} Let $\psi: \mc{K} \rr \C^*$ be an additive character with conductor $\O,$ i.e., $\psi$ is trivial on $\O$ and non-trivial on $\pi^{-1}\O.$  One may easily verify, \be{van-psi} \int_{\O^*} \psi(\pi^k u) du =  \begin{cases} 1 - q^{-1} & \text{ if } k > -1 \\ - q^{-1} & \text{ if } k=-1 \\ 0 & \text{ if } k \leq -2 \end{cases} \ee


\spoint \label{s:hilb-sym} Let $n > 0$ be a positive integer, and denote by $\bmu \subset \mc{K}$ the set of $n$-th roots of unity. We assume first of all that we are in the tame case, so that $(q, n)=1$ where $q$ is the size of the residue field. Furthermore, we assume that $| \bmu | =n.$ This implies that $q \equiv 1 \mod n$, but we make the stronger assumption that $q \equiv 1 \mod{2n}.$ Fix an embedding $\iota:\bmu \hookrightarrow \C^*$ throughout.  The $n$-th order Hilbert symbol (see e.g. \cite[\S 9.2, 9.3]{ser:lf}) is a bilinear map $( \cdot, \cdot) : \mc{K}^* \times \mc{K}^* \rr \bmu$ satisfying \be{hilb-sym} \begin{array}{cc} (x, y) = (y, x)^{-1}, &   (x, -x) = 1 \end{array}.\ee Our assumptions on $q$ and $n$ imply that $(\cdot, \cdot)$ is a power of the usual tame symbol. One can verify that $(\cdot, \cdot)$ is \emph{unramified}, i.e. $(x, y)=1$ if $x, y \in \O^*,$ and also that $(\pi^a,\pi^b)=1$ for any $a,b\in \ZZ.$\footnote{This fact is used in \S \ref{spoint:a-dec}. It requires that $q-1$ be divisible by $2n$ and not just $n$. When this assumption on $q$ and $n$ is not made the structure of the metaplectic torus is significantly more complex, see \cite[3.3]{gg}, and also \ref{tpoint:assumptions} for further remarks on our assumptions.}


\spoint Let $\sigma: \O^* \rr \C^*$ be a multiplicative character. We say that $\sigma$ has conductor of size $a$ (an integer) if $1 +\pi^a\O$ is the largest subgroup on which $\sigma$ is trivial. Let $\tau: \O \rr \C^*$ be an additive character; we say that $\tau$ has conductor of size $b$ (also an integer) where $b= \inf\{ m \in \zee \mid \tau|_{\pi^m \O} \equiv 1 \}.$ We define the \emph{Gauss sum} \be{gauss:gen} \g(\sigma, \tau) = \int_{\O^*} \sigma(u') \tau(u') du' \ee where $du'$ is the restriction of the Haar measure on $\mc{K}$ which gives $\O^*$ volume $q-1.$ Note that if $du$ is the restriction of the usual Haar measure on $\mc{K}$ assigning $\O$ volume $1,$ 
\be{gauss-gen-norm} q^{-1} \g(\sigma, \tau) = \int_{\O^*} \sigma(u) \tau(u) du. \ee 

\begin{nlem}[Lemma 7-4, \cite{ram}] \label{g-gen:van} If $\sigma: \O^* \rr \C^*$ is a multiplicative character with conductor of size $a$ and $\tau:\O \rr \C^*$ is an additive character of conductor of size $b,$ then $\g(\sigma, \tau)=0$ if  $a < b.$ \end{nlem}

\tpoint{Example} \label{ex:gsum} The main examples we shall consider are $\sigma(u)= (u, \pi)^{-k}$ with $(\cdot, \cdot)$ the Hilbert symbol introduced above and $\tau(u)= \psi(-\pi^m u)$ with $\psi$ a principal character as in \S \ref{psi-sec} and $m \in \zee$. Note that $\tau$ has conductor of size $m$ and $\sigma$ has conductor of size $1.$ Hence by Lemma \ref{g-gen:van} we have \be{gvan} \int_{\O^*} (u,\pi)^{-k} \psi(-\pi^{-b} u ) du = 0 \text{ for } b < -1. \ee We set  $\mathbf{g}_{k} := \int_{\O^*} (u', \pi)^{-k} \psi(-\pi^{-1} u' ) du',$ so that  \be{g-2} q^{-1} \mathbf{g}_{k} = \int_{\O^*} (u,\pi)^{-k} \psi(-\pi^{-1} u ) du. \ee We record some facts (also verified in \cite{ram}) : $\mathbf{g}_{k} =\mathbf{g}_{l} $ if $n\mid k-l,$ $\mathbf{g}_{0}=-1,$ $n\nmid k,$ then $\mathbf{g}_{k} \mathbf{g}_{-k}=q.$ Note that we use the assumption $q\equiv 1 \mod \ 2n$ for the last fact. 

\section{Recollections on Metaplectic Groups}

\spoint For any root system $\Phi$ we denote by $R$ the set of roots, $R^{\vee}$ the set of coroots, and $\Pi$ (resp. $\Pi^{\vee}$) the set of simple roots (resp. coroots). Positive and negative roots will be denoted by a $\pm$ subscript, e.g. $R^{\vee}_-$ denotes the negative coroots. Let $W$ be the Weyl group of $\Phi$, and for each $a \in \Pi$ let $w_a$ be the corresponding simple reflection. Let $Q$ (resp. $Q^{\vee}$) denote the root (resp. coroot) lattice, and $\Lambda$ (resp. $\Lambda^{\vee}$) the lattice of weights (resp. coweights).  Denote by $\la \cdot ,\cdot \ra$ the natural pairing $\Lambda^{\vee} \times \Lambda \rr \C.$

\spoint \label{fingp} Let $G$ be a split, simple, simply connected group over $\mc{K}$ with maximal split torus $A,$ and denote by $\Phi(G, A)$ the corresponding root system (note $\Lv = Q^{\vee}$ under our assumptions). Pick a pair of opposite Borel subgroups $B$ and $B^-$ such that $B \cap B^- = A$ For each $a \in R$  there exists a one parameter subgroup $U_a \subset G$ whose elements will be denoted by $\sx_a(s)$ for $s \in \mc{K}.$ For each $a \in R$ and $t \in \mc{K}^*,$ we also set \be{w; h} \begin{array}{lcr} \sw_a(t) = \sx_a(t) \sx_{-a}(-t^{-1})\sx_a(t) & \text{ and } & \sh_a(s) = \sw_a(s) \sw_a(1)^{-1} \end{array}. \ee The group $G$ is generated by $U_a$ for $a \in R$ with relations as described in \cite{steinberg}.   The abelian subgroup $A$ is generated by the elements $\sh_a(s)$ for $a \in R, s \in \mc{K}^*.$ We also denote by $U$ (resp. $U^-$) the unipotent radical of $B$ (resp. $B^-$). Recall that we have the Bruhat decompositions, \be{bru} G = \sqcup_{w \in W}  B \dot{\sw} B = \sqcup_{w \in W} B \dot{\sw} B^- \ee where if $w \in W$ has a reduced decomposition $w= w_{b_1} \cdots w_{b_r}$ with $b_i \in \Pi,$ we let $\dot{\sw}=\sw_{b_1}(-1) \cdots \sw_{b_r}(-1)$ be a fixed lift.   

The maximal compact subgroup will be denoted by $K$ and is equal to the group generated by $\sx_a(u)$ for $a \in R, u \in \O.$ Note that the lifts from the previous paragraph $\dot{\sw}$ lie in $K.$ We denote by $I, I^- \subset K$ the Iwahori subgroups with respect to $B$ and $B^-$ respectively, i.e., \be{iwa-def} \begin{array}{lcr} I = \{ g \in K \mid g \mod{\pi} \in B_{\kappa} \} & \text{ and }  & I^- = \{ g \in K \mid g \mod{\pi} \in B^-_{\kappa} \} \end{array} \ee where the subscript $\kappa$ denotes the same group over the residue field, and $g \mod{\pi}$ denotes the image of $g$ under the natural map $K \mapsto G_{\kappa}.$ We then have \be{K:im} K = \sqcup_{w \in W} I^- \dot{\sw} I = \sqcup_{w \in W} I \dot{\sw} I. \ee  Let $\aw:= \Lv \rtimes W$ denote the affine Weyl group; a typical element of $\aw$ will be written as $(\mv, w)$ with $\mv \in \Lv, w \in W.$ We then have decompositions \be{iwa:dec} G = \sqcup_{x \in \aw} I^- \dot{\sx} I = \sqcup_{x \in \aw} I \dot{\sx} I \ee where $\dot{\sx}= \pi^{\mv} \dot{\sw}$ (with $\dot{\sw}$) as above is our chosen lift of $x$ to $G.$

\spoint \label{spoint:univ:ext} Let $\mc{E}$ be the universal central extension of $G$ as constructed by Steinberg (see \cite{steinberg}). Recall that $\bE$ is generated by symbols $\bx_a(s)$ for $a \in R, s \in \mc{K}$ subject to the relations again described in \emph{loc. cit}. There is a natural homomorphism $p: \bE \rr G$ which sends $\mathbf{x}_a(s) \mapsto \sx_a(s)$ and whose kernel we denote by $C:=\ker(p);$ we have an exact sequence \be{univ:ext} 0 \rr C \rr \bE \stackrel{p}{\rr} G \rr 1. \ee The group $C$ can be described explicitly following Steinberg and Matsumoto as follows. For each $a \in R$ and $t \in \mc{K}^*$ we denote by \be{w; h} \begin{array}{lcr} \bw_a(t) = \bx_a(t) \bx_{-a}(-t^{-1}) \bx_a(t) & \text{ and } & \bh_a(s) = \bw_a(s) \bw_a(1)^{-1} \end{array}. \ee Then for each $a \in R$ and $s, t \in \mc{K}^*$ there exist central elements $\c_{\av}(s, t) \in \mc{E}$ so that \be{h:cen} \bh_a(s) \bh_a(t) \bh_a(st)^{-1} = \c_{\av}(s, t). \ee From Steinberg's work, one deduces (a) for every short coroot $\av \in R^{\vee}$ (or long root) the elements $\c_{\av}(s, t)$ are equal-- denote this value simply by $\c(s, t);$ and (b) there is a $W$-invariant quadratic form $\Q: \C[\Lv] \rr \zee$ so that \be{c:Q} \c_{\av}(s, t) = \c(s, t)^{\Q(\av)}. \ee Let $\bA$ be the subgroup generated by $\bh_a(s)$ for $a \in R, s \in \mc{K},$ and let $\B$ be the bilinear form associated to $\Q,$ namely $\B(x, y)= \Q(x+y) - \Q(x) - \Q(y).$ One can then show, \be{comm:B} [ \bh_a(t), \bh_b(u) ] = \c(u, t)^{\B(\av, \bv)}. \ee Further, Matsumoto has shown \cite{mat} that $C$ is generated by $\sc(s, t)$ for $s, t \in \mc{K}^*$ subject to the relations \be{mat-rel}\begin{array}{lccr} \c(s, -s)=\c(1-s, s)=1, & \c(s, t)= \sc(t, s)^{-1} & \c(s, t)\c(s, u) = \c(s, tu) \end{array}. \ee 

\spoint Pushing forward \eqref{univ:ext} by the Hilbert symbol $(\cdot, \cdot)$ we obtain a group $\tG$ fitting into  \be{met:ext} 0 \rr \bmu \rr \tG \stackrel{p}{\rr} G \rr 1. \ee We shall denote the image of $\bx_a(s)$ simply as $x_a(s),$ and we similarly define $h_a(s)$ and $w_a(s).$ Let $\tU_a$ (resp. $\tU_{a}(\O)$) denote the subgroup generated by $x_a(s)$ for $s \in \mc{K}$ (resp. $s \in \O$), and $\tU, \tU^-$ be the subgroups generated by $\tU_{a}$ for $a \in R_+$ and $a \in R_-$ respectively. Note that $\tU \cong U, \tU^- \cong U^-$ and so we often drop the $\, \tilde{} \, $ and identify $\tU, \tU^-$ with their images $U, U^-$ in $G.$ Let us also set $c_{\av}(s, t)= h_a(s) h_a(t) h_a(st)^{-1}.$ The analogue of the relations \eqref{c:Q} and \eqref{comm:B} persist with $\c_{\av}(s, t)$ replaced by $c_{\av}(s, t).$ Often we write simply $(s, t)$ to denote $c_{\av}(s, t)$ for $\av$ the short coroot. Assuming that $(q, n)=1,$ there exists a splitting $\mathsf{s}$ over $K \subset G.$  Denote by $\tK= \mathsf{s}(K).$ The argument in \cite[Prop 3.3]{savin} shows that $\tK$ is generated by $\tU_a(\O)$ for $a \in R.$ Note that if we have $K=[K,K],$ the splitting $\mathsf{s}$ is also unique. 

\spoint\label{spoint:a-dec} Let $\tA \subset \tG$ be the subgroup generated by $h_a(s)$ for $a \in R, s \in \mc{K},$ and let $\tA_{\O} \subset \tA$ be the subgroup generated by elements $h_a(s)$ for $s \in \O^*$ and $a \in R.$  Note that $\tA_{\O} \subset \tK$ and also, since $(\cdot, \cdot)$ was assumed unramified, $\tA_{\O}$ is abelian. It is easy to see that each $a \in \tA$ admits a decomposition \be{a-dec} a = a' \pi^{\lv} \zeta \text{ where } a' \in \tA_{\O}, \lv \in \Lv, \zeta \in \mu. \ee The element $\lv$ is determined uniquely from $a$ as one can verify by reducing to the non-metaplectic setting. Moreover, since $\bmu$ is a central subgroup, we may use the commutation relations (\ref{comm:B}) to verify that $\tA_{\O} \cap \bmu = 1.$ Hence in \eqref{a-dec} the terms $a', \lv, \zeta$ are uniquely determined from $a.$ With these conventions, set  \be{log:cen} \begin{array}{lcr} \ln (a) := \lv & \text{ and } & \z(a) := \zeta \end{array}, \ee and note these maps descend to the quotient $\tA_{\O} \setminus \tA.$ 

Let $A_0:= C_{\tA}(\tA_{\O})$ denote the centralizer of the $\tA_{\O}.$ Under the assumption that $q \equiv 1 \mod 2n,$ one can show (see \cite[Lemma 5.3]{mac-prin}) that $A_0$ is a maximal Abelian subgroup of $\tA.$ In \S \ref{B:Q}, we study the lattice $\Lv_0$ such that $\Lv_0 \cong A_0/ \bmu \cdot \tA_{\O}.$

%

%
%
%
%
%
\tpoint{Iwasawa and Iwahori-Matsumoto Decompositions} \label{decs} The Iwasawa decomposition for the group $G$ states that $G = K A U$ where if $g \in G$ is written as $g = k a u$ with $k \in K, a \in A,$ and $u \in U$ then $a$ is not well-defined but its class in $ A \cap K \setminus A = A_{\O} \setminus A  = \Lv$ is well-defined. One can lift this to the group $\tG$ as follows: we have \be{iwa-met} \tG = \tK \tA U \ee so that every $g \in \tG$ may be written (non-uniquely) as \be{iwa-elt} g = k a u \text{ where } k \in \tK, a \in \tA, u \in U, \ee with the class of $a$ in $\tA \cap \tK  \setminus \tA$ being well-defined. Denote this class as $\iw_{\tA}(g) \in \tA \cap \tK \setminus \tA,$ and set \be{log:cen:g} \begin{array}{lcr} \ln (g):= \ln(\iw_{\tA}(g)) \in \Lv & \text{ and } & \z(g) :=  \z(\iw_{\tA}(g)) \in \bmu \end{array}. \ee Hence, for each $\nu \in \Lambda$ we may define a function $\Pshi_{\nu}: \tG \rr \C[\Lv]$ as \be{Phi} \Pshi_{\nu}(g) =  q^{\la \nu, \ln(g) \ra} e^{\ln(g) } \iota(\z(g)) = q^{\la \nu, \lv \ra} e^{\lv } \iota(\zeta) \ee where $g \in \tG$ is decomposed according to \eqref{a-dec}, \eqref{iwa-elt}. 

If we define $\tI, \tIm \subset \tK$ in analogy with $I, I^-$ as above, then again we have \be{im-met} \tG = \sqcup_{x \in \aw} \tI \dot{x} \tI = \sqcup_{x \in \aw} \tI \dot{x} \tIm \ee where $\dot{x} \in \tG$ is the lift defined as in \S \ref{fingp} with $\sw_a$ replaced by $w_a.$ The analogue of \eqref{K:im} also holds.  

\newcommand{\gf}{\mathbb{g}}

\section{Metaplectic Demazure-Lusztig Operators}
\label{sec-met-dl} 
In this section, we recall the main constructions of \cite{cgp} in the form which we need for the sequel.  Fix a positive integer $n$ throughout-- it will correspond to the degree of the metaplectic cover. The constructions in \cite{cgp} also makes use of a set of formal parameters $v,$ $\mathbb{g}_0,\ldots ,\mathbb{g}_{n-1}$ satisfying 
\be{par-cond}\mathbb{g}_0=-1\text{ and } \mathbb{g}_i\mathbb{g}_{n-i}=v^{-1}\text{ for }i=1,\ldots ,n-1.\ee
In the \emph{sequel} (i.e. when we move to the $p$-adic setting), we will set 
\be{par:choices} v=q^{-1}\text{ and }\gf_k=\mathbf{g}_{k}\text{ for }k=0,\ldots ,n-1.\ee
where $\mathbf{g}_{k}$ is the Gauss sum defined in \S \ref{ex:gsum}, but in this section we treat $v$ and the $\gf_i$ as formal variables satisfying the above conditions.   Note that the identities in the present section, in particular \eqref{P:cs} and Corollary \ref{CG-is-action} hold in the formal setting \cite{ch:gu, cgp}. However, the proof of Corollary \ref{CG-is-action} presented in \S \ref{braid-cg} only addresses the case when the choices of \eqref{par:choices} are made.

\spoint \label{rings} Let $\C[\Lv]$ denote the group algebra of the coweight lattice. To each $\lv \in \Lv$ we denote by $e^{\lv}$ the corresponding element in $\C[\Lv]$ with multiplication defined as $e^{\lv} e^{\mv} = e^{\lv + \mv}$ for $\lv, \mv \in \Lv.$ The natural action of $W$ on $\Lv$ (i.e., $w.e^{\lv}= e^{w \lv}$ for $w \in W, \lv \in \Lv$) extends linearly to an action of $W$ on $\C[\Lv].$ For $h \in \C[\Lv], w \in W$ we sometimes write $h^w$ for the element obtained by applying $w$ to $h.$


\spoint \label{B:Q} In \S \ref{spoint:univ:ext}, we introduced a $W$-invariant quadratic form $\Q: \Lv \rr \zee$ with associated bilinear form $\B(\cdot, \cdot).$  The $W$-invariance of $\Q$ is equivalent to the following fact (see \cite[(4.5.1)]{del-bry}): for any $a \in R$ we have \be{w-inv} \B(\lv, \av)= \la \lv, a \ra \Q(\av). \ee Let $\Lv_0 \subset \Lv$ be the sublattice defined as \be{Lv0} \Lv_0:= \{ \lv \in \Lv \mid \B(\lv, \av) \equiv 0 \mod{n} \text{ for all } \av \in \Pi^{\vee} \}. \ee For each $a \in R,$ define an integer \be{ma:def} \m(\av):= \frac{n}{\gcd(n, \Q(\av))}, \ee and the residue map \be{resid} \resi_{\m(\av)}: \zee \rr \{ 0, 1, \ldots, \m(\av)-1 \}. \ee  It is easy to verify the following  properties: {\center \begin{enumerate}
  \item[(i)] For any integer $k \in \zee$ we have $k \Q(\av) \equiv 0 \mod n$ if and only if  $k \equiv 0 \mod \m(\av)$
  \item[(ii)] For any $\av \in R^{\vee}$ we have $\m(\av)=\m(w \av)$ for $w \in W.$
  \item[(iii)] For any $a \in R$ we have $\m(\av) \av \in \Lv_0$ (this follows from \eqref{w-inv}).
  \end{enumerate}
}

\spoint \label{subs:CGaction} Set $\C_v := \C[v, v^{-1}]$ and denote by $\C_v[\Lv]:= \C_v \otimes_{\C} \C[\Lv].$ We shall also need some localizations of this ring in the sequel. Let $\mult$ be the smallest subset of $\C_v[\Lv _0]$ closed under multiplication containing $1 - e^{-\m(\av) \av}$ and $1 - v e^{-\m(\av) \av}$ for every $a\in R,$ and let $\ring$ (respectively, $\nice$) denote the localization of $\C_v[\Lv]$ (respectively, of $\C_v[\Lv_0]$) by $\mult .$ 

Following Chinta and Gunnells \cite{ch:gu}, for $a \in \Pi$ we set  
{\small \be{cg-act} w_a \star e^{\lv} = \frac{e^{w_a \lv}}{1 - v e^{-\m(\av) \av}} \left[ (1-v) e^{ \resi_{\m(a)} \left( \frac{ B(\lv, \av)}{\Q(\av)} \right) \av } - v \gf_{ \Q(\av) + \B(\lv, \av) } e^{ (\m(\av) - 1) \av} (1 - e^{-\m(\av) \av}) \right].\ee} Extend by $\C_v$-linearity to define $w_a \star f$ for every $f \in \C_v[\Lv],$ and use the formula \be{cg-act-quot} w_a \star \frac{f}{h} = \frac{w_a \star f}{h^{w_a}} \text{ for } f \in \C_v[\Lv], h\in \mult, \ee to extend this to $w_a \star -: \ring \rr \ring.$


\begin{nlem} \label{cg-basic-prop} Let $f \in \C_v[\Lv]$ and $a \in \Pi.$ Then 
\begin{enumerate}
\item We have $w_a \star (h f) = h^{w_a} (w_a \star f)$ for $h\in \nice$ (note: this fails for general $h$).
\item $w_a \star: \ring \rr \ring$ is an involution: i.e., $w_a \star w_a (f) = f.$  
\end{enumerate}
\end{nlem}

\begin{proof} The proof of the first statement follows immediately from the definitions.  The proof of the second is a straightforward computation, making repeated use of the first part and \eqref{par-cond}. \end{proof}

\spoint For each $a \in R$ we introduce two important rational functions 
\be{c:b} \begin{array}{lcr} \b(\av):= \frac{v-1}{1 - e^{\m(\av) \av}} & \text{ and } & \c(\av):= \frac{1 - v e^{-\m(\av) \av}}{1 - e^{\m(\av)\av }} \end{array}, \ee and define, for each $a \in \Pi$ the operators  \be{T:c-b} \T_a(e^{\lv}) = \c(\av) w_a \star e^{\lv} + \b(\av) e^{\lv}. \ee Replacing $w_a \star e^{\lv}$ with $e^{w_a \lv}$ we obtain the operators in the non-metaplectic setting which appeared in \cite{pat-whit, bbl}. Let us also record here the following simple identity, \be{c:act} \c(\av) w_a \star e^{\lv} = \frac{1-v}{1 - e^{\m(\av) \av}} e^{w_a \lv} \, e^{( \resi_{\m(\av)} \la \lv, a \ra)  \av}  + \,  v \gf_{(1 + \la \lv, a \ra) \Q(\av)} e^{w_a \lv - \av} . \ee We also define $\T_1: \C_{v}[\Lv] \rr \C_{v}[\Lv]$ to be the identity operator.  \begin{nclaim} \label{T:pol} For any $a \in \Pi$ we have $\T_a: \C_v[\Lv] \rr \C_v[\Lv].$ \end{nclaim} \begin{proof} Since $\T_1(e^{\lv}) = e^{\lv}$ it suffices to show that $\T_a + \T_1$ preserves $\C_{v}[\Lv].$ We have three cases to consider: (i) $\la \lv, a \ra >0$; (ii) $\la \lv, a \ra < 0$; and (iii) $\la \lv, a \ra =0.$ In case (iii) we see that \be{iii} \T_a(e^{\lv}) + \T_1(e^{\lv})= e^{\lv} - v \gf_{\Q(\av)} e^{\lv - \av} \ee so the result follows. As (i) and (ii) are similar, we just focus on (i). In this case,  \be{rk1:T} (\T_a +\T_1)(e^{\lv})  = \underbrace{\frac{ 1 - v e^{-\m(\av) \av}}{1 - e^{-\m(\av) \av}}  e^{\lv} }_{(I)} + \underbrace{ \c(\av) w_a \star e^{\lv}}_{(II)}.  \ee Expanding the rational function appearing in (I) in powers of $e^{-\m(\av) \av}$ we see that \be{1} (I) = e^{\lv} ( 1 + (1-v) e^{-\m(\av) \av} + (1-v) e^{-2 \m(\av) \av} + \cdots ). \ee On the other hand, from \eqref{c:act} we have \be{2} (II) &=& \frac{1-v}{1 - e^{\m(\av) \av}} e^{w_a \lv} \, e^{( \resi_{\m(\av)} \la \lv, a \ra)  \av}  + \,  v \gf_{(1 + \la \lv, a \ra) \Q(\av)} e^{w_a \lv - \av} \\
&=&  \frac{v-1}{1 - e^{\red{-}\m(\av) \av}} e^{w_a \lv- \m(\av) \av} \, e^{( \resi_{\m(\av)} \la \lv, a \ra)  \av}  + \,  v \gf_{(1 + \la \lv, a \ra) \Q(\av)} e^{w_a \lv - \av} \ee Expanding the rational function appearing in the first term in powers of $e^{-\m(\av)\av},$ and keeping track of the coefficients in front of these powers, one easily sees that $(I) - (II)$ is \be{I-II}  e^{\lv} +  \sum_{\substack{k > 0 \\ k \m(\av) \leq \la \lv, a \ra \\ }} (1-v) e^{\lv - k \m(\av) \av} \ +  \  v \gf_{\Q(\av) + \B(\lv, \av) } e^{w_a{\lv} - \av}. \ee In sum: if $\la \lv, a \ra >0,$ we have \be{Ta:i} \T_a(e^{\lv}) = \sum_{\substack{k > 0 \\ k \m(\av) \leq \la \lv, a \ra \\ }} (1-v) e^{\lv - k \m(\av) \av} \ +  \  v \gf_{\Q(\av) + \B(\lv, \av) } e^{w_a{\lv} - \av}, \ee and so case (i) is proven. \end{proof} 

\spoint \label{w-defs} Let $w \in W$ and choose any reduced decomposition $w=w_{b_1} \cdots w_{b_r}$ with $b_i \in \Pi$ for $i=1, \ldots, r.$ Then we may define operators $w \star - : \nice \rr \nice$ and $\T_w: \C_v[\Lv] \rr \C_v[\Lv]$ as follows, \be{w-def} w \star f &:=& w_{b_1} \star \cdots \star w_{b_r} \star f \text{ for } f \in \nice  \\ \label {Tw:def} \T_w (f) &:=& \T_{b_1} \cdots \T_{b_r} (f) \text{ for } f \in \C_v[\Lv] \ee but it remains to be seen that these definitions are well-defined (i.e. independent of the choice of reduced decompositions-- see Theorem \ref{well-def} for the precise statement). To verify this, one may proceed as in \cite{ch:gu, cgp}; an independent proof of this fact (under the specialization \eqref{par:choices}) is given in the Appendix.  Since $w_a$ also acts as an involution (Lemma \ref{cg-basic-prop} (2) ) on $\nice$, from the standard presentation of $W$ we obtain

\begin{ncor} \label{CG-is-action} The operation $\star$ defines an action of $W$ on $\nice .$ \end{ncor}

Note that this result is not used in the proof of Theorem \ref{main-rec}.

%
%
%
%
%
%
%

\spoint \label{sym-fla} The results in this paragraph assume Theorem \ref{well-def} and Corollary \ref{CG-is-action}. Consider the symmetrizer \be{sym:P} \mc{P}:= \sum_{w \in W} \T_w, \ee which we can regard as an operator on $\C_v[\Lv].$ Let us also define 
\be{w-negate} R^{\vee}(w)=\{\av\in R^{\vee}_+\mid w\av \in R^{\vee}_-\},\ee
\be{del:v} \begin{array}{lcr} \Delta_v:= \prod_{a \in R_{+}^{\vee}}  (1 - ve^{-\m(\av) \av})& \text{ and set } &  \Delta= \Delta_1. \end{array} \ee\begin{nthm} \label{cgp-main} \cite[Theorem 4]{cgp} If $\lv \in \Lv_+,$ we have {\small \be{P:cs} \mc{P}(e^{\lv}) = \frac{\Delta_v}{\Delta}  \ \sum_{w \in W} (-1)^{\ell(w)} \left( \prod_{\bv \in R^{\vee}(w^{-1}) } e^{-\m(\bv) \bv} \right) w \star e^{\lv}, 
\ee } where $\star$ is the action introduced in \eqref{cg-act}.  \end{nthm}

\tpoint{Remarks} In the rank one case, the identity follows immediately from \eqref{rk1:T}; in the non-metaplectic case, the above result appears essentially as \cite[\S 5.5]{mac:aff}, after using the relation \cite[(6.6)]{pat-whit}. 
The proof of the above Theorem (see \cite[Theorem 4]{cgp}), which makes use of the long word of the Weyl group, can also be viewed in the following way. First, using methods similar to \cite[Lemma 6.4, Proposition 6.5]{pat-whit} (modified suitably using the metaplectic identities of \cite{cgp}), one may prove a slightly weaker statement without mention of the long word, 
 \be{P:cs-op} \mc{P}(e^{\lv}) = {\mathfrak{m}}\cdot \frac{\Delta _v}{\Delta } \sum_{w \in W} (-1)^{\ell(w)} \left( \prod_{\bv \in R^{\vee}(w^{-1}) } e^{-\m(\bv) \bv} \right) w \star e^{\lv}  ,\ee
where the factor $\mathfrak{m} \in \nice $ is some $W$-invariant element. Using the long word, one can then show $\mf{m}=1.$ See also Remark \ref{whit:spy}.

\section{The Iwahori-Whittaker Functions and Statement of Main Result}

In this section, we formulate the main result of this paper and explain its proof in the rank 1 case. 

\spoint The Whittaker function is defined via the integral \be{whit-def} \W(g) = \int_{U^-} \Pshi_{\rho}(g u^-) \psi(u^-) du^-, \text{ for } g \in \tG \ee where $du^-$ is the Haar measure on $U^-$ assigning $U^-_{\O}$ volume 1, and $\Pshi_{\nu}$ is defined for any $\nu \in \Lambda$ in \eqref{Phi}.  The same argument as in the non-metaplectic case shows this integral is a well-defined element in $\C[\Lv],$ i.e., $\W: \tG \rr \C[\Lv].$ From the definition of $\Pshi_{\rho}$, one has $\W(\zeta g) = \iota(\zeta) g$ for $g \in \tG, \zeta \in \bmu$ (recall $\iota$ was a fixed embedding of $\bmu$ in $\C^*$)  Together with the following standard fact, we see that $\W(g)$ is determined by the values $\W(\pi^{\lv})$ for $\lv \in \Lv_+.$

\begin{nprop} \label{p:W-inv} The function $\W(g)$ satisfies the following transformation rule, \be{W-inv} \W(k g u^- ) = \psi(u^-)^{-1} \W(g). \ee Moreover, if $g = \pi^{\lv}$ with $\lv \notin \Lv_+$ then $\W(\pi^{\lv})=0.$ \end{nprop}

A simple change of variables \cite[Eq. (2.15)]{pat-whit} shows that \be{W-cov} \W(\pi^{\lv}) = \Pshi_{-\rho}(\pi^{\lv}) \int_{U^-} \Pshi_{\rho}(u^-) \psi_{\lv}(u^-) du^-, \ee where $\psi_{\lv}(x) = \psi(\pi^{-\lv} x \pi^{\lv})$ for $x \in U^-.$  

\spoint If $\Gamma$ is any group and $X$ is a right $\Gamma$-set and $Y$ is a left $\Gamma$-set, we define \be{fiber} X \times_{\Gamma} Y =  X \times Y  / \sim \ee where $\sim$ is the equivalence relation generated by $(x\gamma, \gamma^{-1} y) \sim (x, y)$ for $x \in X, y \in Y,$ and $\gamma \in \Gamma.$ For each $w \in W, \lv \in \Lv$ multiplication induces a natural map \be{m_w} m_{w, \lv}: \bmu U \dot{w} \tI^- \times_{\tI^-} \tI^- \pi^{\lv} U^- \rr \tG. \ee Let $\lv \in \Lv_+$ and $\mv \in \Lv.$ For each $x \in m_{w, \lv}^{-1}(\pi^{\mv}),$ there are natural projections \be{proj-fib} \begin{array}{lcr} \nn: m_{w, \lv}^{-1}(\pi^{\mv}) \rr U^-_{\O} \setminus U^- & \text{ and } &  \z_w: m_{w, \lv}^{-1}(\pi^{\mv}) \rr \bmu \end{array} \ee defined as follows. Let $x \in m_{w, \lv}^{-1}(\pi^{\mv})$ have representative $(a, b)$ with $a \in \bmu U \dot{w} \tI^-$ and $b \in \tI^- \pi^{\lv} U^-$ some chosen representatives. Writing $b = i \pi^{\lv} u^-$ with $i \in \tI^-,$ $u^- \in U^-$ one can verify that since $\lv \in \Lv_+$ the class of $u^- \in U^-_{\O} \setminus U^-$ is well-defined; this class is denoted by $\nn(x).$  Writing $a= \zeta u \dot{w} i$ with $\zeta \in \bmu,$ $u \in U^-,$ $i \in \tI^-,$ the element $\zeta$ is well-defined and independent of the representatives $(a, b)$ taken for $x.$ This $\zeta\in \bmu$ is denoted by $\z_w(x).$ Using these maps, we can define the \emph{Iwahori-Whittaker functions} as the formal generating series \be{formal-w} \W_{w, \lv} := \sum_{\mv \in \Lv} e^{\mv} q^{ \la \rho, \mv \ra} \left( \sum_{x \in m_{w, \lv}^{-1}(\pi^{\mv}) } \psi(\nn(x)) \iota(\z_w(x)) \right). \ee The same argument as in \cite[Lemma 4.4]{pat-whit} shows that we have a decomposition \be{W:dec} \W(\pi^{\lv}) =  q^{-\la 2 \rho, \lv \ra} \sum_{ w \in W} \W_{w, \lv}, \ee where we note that the factor of $q^{-\la 2 \rho, \lv \ra}$ must be introduced here for the same reason as \cite[\S 2.7]{pat-whit}.

\tpoint{Remark} Suppose $w=1,$ and pick $\mv \in \Lv$ such that $m_{1, \lv}^{-1}(\pi^{\mv}) \neq \emptyset.$ So $\pi^{\mv} \in \bmu U \tI^- \pi^{\lv} U^-,$ which forces $\mv=\lv,$ and hence $\W_1(\pi^{\lv})$ is a scalar multiple of $e^{\lv}.$ An easy computation shows $m_{1, \lv}^{-1}(\pi^{\lv})$ consists of a single element and so $\W_1(\pi^{\lv}) = q^{ \la \rho, \lv \ra} e^{\lv}.$ We can also write this as 
\be{W:1} \W_1(\pi^{\lv}) = \Pshi_{\rho}(\pi^{\lv} )\int_{U^-_{\O}} \Pshi_{\rho}(u^-) \psi_{\lv}( u^-) du^-, \ee using the fact that $\lv \in \Lv_+.$ In fact, more generally one may show 
\be{W:1w} \W_w(\pi^{\lv}) = \Pshi_{\rho}(\pi^{\lv} )\int_{U^{-,w}} \Pshi_{\rho}(u^-) \psi_{\lv}( u^-) du^-, \ee where the sets  \be{U-w} U^{-, w} = \{ x \in U^- \mid x \in \tI^- w^{-1} \tB \} \ee are the analogues of the ones defined in \cite[(2.6)]{pat-whit} (note that what is defined there as $U^{-, w}$ is our $U^{-, w^{-1}}$). Using the Iwasawa decomposition of $G$ and Bruhat decomposition for the group $G_\kk$ over the residue field, one obtains a decomposition $U^- = \sqcup_{w \in W}  U^{-, w}$ which underlies \eqref{W:dec}. In the case of $SL(2)$, we have $U^- = U^{-,1} \sqcup U^{-, w_a}$ where $U^{-,1} = U_{-a, \O}$ and $U^{-, w_a}= U_{-a}(\mc{K})  \setminus U_{-a, \O}$ (where $a$ is the unique simple root). 

\spoint Recall for $w=1$ we defined $\T_1: \C_v[\Lv] \rr \C_v[\Lv]$ as the identity operator. Hence from \eqref{W:1}, we have \be{W:1:T} \W_{1, \lv} = q^{ \la \rho, \lv \ra} e^{\lv} = q^{  \la \rho, \lv \ra} \T_1(e^{\lv}). \ee The main result of this paper is the following,

\begin{nthm} \label{main-rec} Let $w,w' \in W$ and $a \in \Pi$ so that $w= w_a w'$ with $\ell(w) = \ell(w')+1,$ we have \be{T:W-rec} \T_a(\W_{w', \lv}) = \W_{w, \lv} \ee \end{nthm} An induction using \eqref{W:1:T} shows,

\begin{ncor} \label{W:T} Let $w \in W$ and $\lv \in \Lv_+.$ For any reduced decomposition $w=w_{b_1} \cdots w_{b_r}$ with $b_{1}, \ldots, b_{r} \in \Pi,$  \be{W:T-w} \W_{w, \lv} = q^{  \la \rho, \lv \ra} \T_{b_{1}} \cdots \T_{b_{r}}(e^{\lv}). \ee 
\end{ncor}

\tpoint{Remark} \label{rmk:DL-braid} As the left-hand side of \eqref{W:T-w} is independent of the choice of reduced decomposition of $w$, the same can be concluded for the right-hand side: for $w \in W$ and any reduced decomposition $w=w_{b_1} \cdots w_{b_r}$ as above, the expression \be{T:w-def-applied} \T_w(e^{\lv}):=  \T_{b_{1}} \cdots \T_{b_{r}}(e^{\lv}), \, \lv \in \Lv_+ \, \ee  is independent of the reduced decomposition of $w.$ We have not yet defined $\T_w(e^{\mv})$ for a general $\mv \in \Lv;$ in fact, as we show in Appendix \ref{app:braid}, the formula \eqref{T:w-def-applied} considered with an arbitrary $\lv \in \Lv$ yields a well-defined element in $\C[\Lv].$ Then Corollary \eqref{W:T} can be restated as follows, \be{W:T-1} \W_{w, \lv} = q^{  \la \rho, \lv \ra} \T_w(e^{\lv}) \text{ for } w \in W, \lv \in \Lv_+. \ee

\spoint \label{met-cs} In this paragraph, we again assume Corollary \ref{CG-is-action} and Theorem \ref{well-def}, so that the algebraic formula \eqref{P:cs} of Theorem \ref{cgp-main} holds. Using this theorem and \eqref{W:T-1}, we obtain the Metaplectic Casselman-Shalika formula that was obtained using different techniques in \cite{ch:of, mac-met}.

\begin{ncor} \label{main-cs-fin}\cite{ch:of, mac-met} For each $\lv \in \Lv_+,$ we have the equality \be{main-cs-fin-1} \W(\pi^{\lv}) &=& q^{- \la \rho, \lv \ra} \prod_{a > 0}  \frac{1 - q^{-1} e^{-\m(\av) \av}}{1 - e^{- \m(\av) \av} } \, \ \sum_{w \in W} (-1)^{\ell(w)} \left( \prod_{\bv \in R^{\vee}(w^{-1}) } e^{-\m(\bv) \bv} \right) w \star e^{\lv}, 
\ee where the action $\star$ is the one defined in \eqref{cg-act} with $v=q^{-1}.$ \end{ncor}

\tpoint{The rank one case} \label{rk1case}  Let us conclude this section with a proof of Theorem \ref{main-rec} in the case of $SL(2),$ where we have two Iwahori-Whittaker functions $\W_1$ and $\W_{w_a}:= \W_a.$ We have already seen that $\W_1(\pi^{\lv}) =  q^{ \la \rho ,\lv \ra}  \T_1(e^{\lv})$ and so we turn to computing $\W_a(\pi^{\lv}),$ which is given by the following integral, \be{Wa} \label{Wa} \W_{a}(\pi^{\lv})= \Pshi_{ \rho}( \pi^{\lv}) \int_{U_{-a}(< 0)} \Pshi_{\rho} (x_{-a} ) \psi_{\lv}(x_{-a} ) dx_{-a}, \ee where $U_{-a}(< 0) = U_{-a} \setminus U_{-a, \O}.$  For each $k \in \zee,$ let  $U_{-a}[k]= \{ x_{-a}(s) \mid \valu(s) = k \}.$ Then one has   \be{W:sum} \W_a(\pi^{\lv}) &=& \Pshi_{ \rho}(\pi^{\lv}) \sum_{k > 0 }  \int_{U_{-a}[-k]} \Pshi_{\rho} (x_{-a} ) \psi_{\lv}(x_{-a} ) dx_{-a} 
= \Pshi_{\rho}(\pi^{\lv})  \sum_{k > 0 } \mc{I}_a(k), \ee where we have introduced the notation, \be{I-k}  \mc{I}_a(k):= \int_{U_{-a}[-k]} \Pshi_{\rho} (x_{-a} ) \psi_{\lv}(x_{-a} ) dx_{-a}. \ee

\begin{nlem} \label{I-k} With the notations introduced above, we have 
\be{I-k-cases} \mc{I}_a(k):= \begin{cases}  e^{- k \av} (1- q^{-1}) & \text{ if } k \leq \la \lv, a \ra \text{ and } k \Q(\av) \equiv 0 \mod r \\ e^{ -(\la \lv, a \ra+ 1) \av} \, q^{-1} \mf{g}_{ \Q(\av)( \la \lv, a \ra +1)  } & \text{ if } k = \la \lv, a \ra + 1  \\
0 & \text{ otherwise} \end{cases} \ee \end{nlem} 

Assuming the Lemma (which is proved below), we conclude \be{W-rank1} \W_a(\pi^{\lv}) = q^{ \la \rho, \lv \ra } \left(  (1-q^{-1} )  \sum_{\substack{ 1 \leq k \leq \la \lv, a \ra \\ r | k \Q(\av)  }} e^{\lv - k \av} + q^{-1} \mathbf{g}_{\Q(\av)(1 + \la \lv, a \ra) } e^{w_a \lv - \av} \right) = q^{ \la \rho, \lv \ra} \T_a(e^{\lv}) \ee where the second equality uses \eqref{Ta:i}.  Using \eqref{W:dec}, we obtain \be{W-rank1} \W(\pi^{\lv}) = q^{ \la \rho, \lv \ra } \left(  e^{\lv} + (1-q^{-1} )  \sum_{\substack{ 1 \leq k \leq \la \lv, a \ra \\ r | k \Q(\av)  }} e^{\lv - k \av} + q^{-1} \mathbf{g}_{\Q(\av)(1 + \la \lv, a \ra) } e^{w_a \lv - \av} \right) . \ee

\begin{proof}[Proof of Lemma \ref{I-k}]  If $s \in \mc{K}$ with $\valu(s) < 0$ then the identity  \be{u:iw} x_{-a}(s) = x_{a}(s^{-1}) w_a h_a(s) x_a(s^{-1}), \ee provides an Iwasawa decomposition for $x_{-a}(s).$ We also record here that  \be{rk1:h} h_a(s) = h_a(r) \pi^{-k \av} (r, \pi)^{ k \Q(\av)} \text{ where } s= \pi^{-k} r, \,  r \in \O^*.\ee Thus we have \be{tp-1} \Pshi_{\rho}(x_{-a}(s)) = \Pshi(h_a(s)) = q^{-k} e^{-k \av} (r, \pi)^{k \Q(\av)} \ee where $s= \pi^{-k} r$ with $r \in \O^*.$  We have three cases to consider: 

\begin{enumerate}
\item If $\la \lv, a \ra - k \geq0$ then $\psi_{\lv}(x_{-a}(s))=1$ and  \be{case1}  I_a(k) = e^{ - k \av} \int_{\O^*} (r, \pi)^{  k \Q(\av)} dr = \begin{cases} e^{- k \av} (1- q^{-1}) & \text{ if } k \Q(\av) \equiv 0 \mod n \\  0 & \text{ otherwise} \end{cases}  \ee

\item If $\la \lv, a \ra - k =-1,$ then the same computation as in the previous step shows that the integral now reduces to   \be{rk1:k} && e^{ -(\la \lv, a \ra + 1) \av } \int_{\O^*} (u, \pi)^{\left( \la \lv, a \ra +1 \right) \Q(\av)} \psi(-u^{-1}/\pi ) du 
= e^{ -(\la \lv, a \ra+ 1) \av} \, q^{-1} \mf{g}_{\Q(\av)( \la \lv, a \ra +1)  }, \ee where the second equality  follows from \eqref{g-2}.  
 \item $\la \lv, a \ra - k < -1$ then the same computation as in the previous step together with \eqref{gvan} implies the result.
 \end{enumerate} 
\end{proof}

\section{Averaging and Intertwining Operators} 
\label{avg-sec}

\spoint \label{prin-ser} Let $M(\tG)$ be the vector space of functions $f:\tG \rr \C$ such that $f(a_{\O} u g) = f(g)$ for $a_{\O} \in \tA_{\O}, u \in U, g \in \tG$ and $f(\zeta g) = \iota(\zeta) f(g)$ for $\zeta \in \bmu$ and $g \in \tG$ (and recall $\iota: \bmu \hookrightarrow \C^*$ was our fixed embedding). We shall write $M(\tG, \tI)$ (or $M(\tG, \tI^-)$) for the set of $f \in M(\tG)$ which are right $\tI$ (or $\tI^-$) invariant. For the moment, we consider \emph{all} such functions, though in practice we only need to work within a smaller space. 

\spoint \label{prin-ser:2} The space $M(\tG)$ carries an action of $\C[\Lv_0]$: for $\mv \in \Lv_0$ and $f \in M(\tG)$ we set 
\be{act:f} (e^{\mv} f) (g)  = q^{- \la \rho, \mv \ra} f( \pi^{-\mv} g). \ee For $f \in M(\tG)$ and $a \in \Pi$ a simple root, define the intertwining integral \be{Ia:f} I_a(f) ( x ) = \int_{U_a} f(w_a n_a x) dn_a \ee where $dn_a$ is the Haar measure on $U_a$ giving $U_a(\O)$ volume $1.$ Note that the above integral need not be well-defined for every $f \in M(\tG),$ but it will be for the specific functions we consider below.  If $I_a(f)$ is well-defined, one verifies immediately the formula, \be{int-act} I_a(e^{\mv} f ) = e^{w_a \mv} I_a(f)  \text{ for } \mv \in \Lv_0. \ee

\spoint \label{vx} Let $\aw$ be as in \S \ref{fingp}, \ref{decs}. For each $x \in \aw$ let $\tve_x: \tG \rr \C$ be function supported on $\bmu \tA_{\O} U \dot{x} \tI^-$ and defined as 
\be{vew} \tve_x(\zeta g):= \begin{cases} \iota(\zeta) & \text{ if } g \in \tA_{\O} U \dot{x} \tI^-, \zeta \in \mu \\ 0 & \text{ otherwise} \end{cases}. \ee Note that we have $e^{\mv} \tve_w := q^{ - \la \rho, \mv \ra} \tve_{\pi^{\mv} w}$ where the element $\pi^{\mv} w \in \aw$ is regarded as the multiplication of $(\mv, 1)$ and $w$ in $\aw.$ 

\begin{nlem} \label{lem-int-vw} Let $a \in \Pi$ so that $w = w_a w'$ with $\ell(w) = \ell(w')+1$ . Then we have \be{int-vw} I_a(\tve_{w'}) &=& \tve_w + \frac{1- q^{-1}}{1 - e^{ \m(\av) \av} } \tve_{w'} \\
&=& \tve_w + (1-q^{-1}) \sum_{ k \geq 0} q^{-k \m(\av)} \tve_{\pi^{k \m(\av) \av} w'}. \ee  \end{nlem} 

The proof in the non-metaplectic case is standard; the only new complications which arise in the metaplectic case are among those already addressed in the proof of Lemma \ref{I-k}.

\spoint For each $\xv \in \Lv,$ define a function $\hs_{\psi, \xv}: \tG \rr \C^*$ with support on $\bmu U \tA_{\O} \pi^{\xv} U^-$ as follows \be{def-spw} \hs_{\psi, \xv} (g) = \begin{cases} \iota(\zeta) \psi(n^-)&  \text{ if } g = \zeta n a_{\O} \pi^{\xv} n^-, \text{ with } n \in U, n^- \in U^-, a_{\O}  
\in A_{\O}  \\ 0 & \text{ otherwise.} \end{cases} \ee The following is the main fact we need about the functions. For $a \in \Pi$ and $\xv \in \Lv,$ set \be{Gam1} \Gamma^I_{a, \xv} &:=& q^{-1} \,  \g_{(\la \xv, a \ra +1) \Q(\av)} q^{(1 + \la \xv, a \ra)} \hs_{\psi, w_a\xv- \av} \\ \label{Gam2} \Gamma^{II}_{a, \xv} &:=& \sum_{\substack{ m \geq 0 \\ m \equiv \la \xv, a \ra \mod{\m(\av)} }} (1 - q^{-1}) q^{-m + \la \xv, a \ra} \hs_{\psi, w_a\xv+ m \av}. \ee  
\begin{nlem} \label{I_a:spw} For $g \in \bmu U \tA U^-$, we have \be{Ia:Gam} I_a (\hs_{\psi, \xv}) (g)= \Gamma_{a, \xv}(g):= \Gamma^I_{a, \xv} +\Gamma^{II}_{a, \xv} \ee \end{nlem}  
\begin{proof} By the left $U$-invariance and right $(U^-, \psi)$-invariance of both sides, it suffices to compute  \be{Ia-1} I_a(\hs_{\psi, \xv})(\pi^{\mv})=\int_{U_a} \hs_{\psi, \xv}(w_a n_a \pi^{\mv}) dn_a \text{ for } \mv \in \Lv. \ee Recall \eqref{u:iw}, $w_a x_a(s) = x_a(-s^{-1})h_a(s^{-1}) x_{-a}(s^{-1}),$ and if $s=\pi^k r$ then by \eqref{rk1:h}  \be{ha-exp} h_a(s^{-1}) = h_a(r^{-1}) \pi^{- k \av} (r, \pi)^{-k \Q(\av)}. \ee So if $n_a= x_a(s) \in U_a[k],$ and $\mv \in \Lv$ is such that $\hs_{\psi, \xv}(w_a n_a \pi^{\mv}) \neq 0,$ the above formula implies $\mv = \xv + k \av.$ For this choice of $\mv$ an easy computation shows \be{int-k} \int_{U_a[k]} \hs_{\psi, \xv}(w_a n_a \pi^{\mv}) dn_a = q^{-k} \int_{\O^*} (r,\pi) ^{-k \Q(\av)} \psi( -\pi^{ \la \xv, a \ra + k } r^{-1}) dr. \ee As in Lemma \ref{I-k}, we consider three cases,
\begin{enumerate}
\item If $k + \la \xv, a \ra \geq 0,$ then $\mv = \xv + k \av.$ Let us write this as  $\mv = w_a \xv + m \av,$ with $k= m - \la \xv, a \ra.$ The integral \eqref{int-k} assumes the value \be{int-k:a} q^{-k} \int_{\O^*} (r, \pi)^{-k \Q(\av)} dr = \begin{cases} q^{-k} (1- q^{-1}) & \text{ if } n \mid k \Q(\av) \\ 0 & \text{ otherwise} \end{cases} \ee 

\item If $k + \la \xv, a \ra = -1$ (or $\mv = w_a \xv - \av$) then \eqref{int-k} now becomes \be{int-k:b} & & q^{1 + \la \xv, a \ra} \int_{\O^*} (r, \pi)^{(1 + \la \xv, a \ra) \Q(\av)} \psi( -\pi^{-1} r^{-1}) dr  = q^{-1} \cdot q^{1 + \la \xv, a \ra} \g_{(1 + \la \xv, a \ra) \Q(\av) }, \ee where we used \eqref{g-2} in the last line.

\item If $k + \la \xv, a \ra < -1,$ then the corresponding integral vanishes again using \eqref{gvan}.
\end{enumerate} \end{proof}

\tpoint{Braid Relations for the Chinta-Gunnells Action} \label{braid-cg} Let us introduce the notation $X^{\xv}:= \hs_{\psi, \xv}$ for $\xv \in \Lv,$ and consider the ring $V:= \C[X^{\xv}, \xv \in \Lv]$  where $X^{\xv}X^{\mv}=X^{\xv+\mv}$ for $\xv, \mv \in \Lv,$ i.e., $V \cong \C_v[\Lv]/(v-q^{-1}).$ As in \S \ref{subs:CGaction}, we may construct a localization $V_{S}$ of this ring, so that $V_S \cong \C_{v, S}[\Lv]/(v-q^{-1})$ With this definition, one can check that $I_a(X^{\xv}) \in V_S.$

For $\lv \in \Lv_0$ (and so $\pi^{\lv} \in A_0$) we have 
\be{act:lv0} e^{\lv} \hs_{\psi, \xv} = q^{ - \la \rho, \lv \ra} \hs_{\psi, \xv + \lv } \text{ for } \lv \in \Lv_0. \ee and so we may define an action of $\C[\Lv_0]$ (and some localizations of it) by specifying 
\be{newact} e^{\lv} \circ X^{\xv} := q^{ - \la \rho, \lv \ra} X^{\xv + \lv} \text{ for } \lv \in \Lv_0, \xv \in \Lv. \ee  From \eqref{int-act} and the stability of $\Lv_0$ under the usual $W$-action, we find \be{Ia:act} I_a( e^{\lv} \circ X^{\xv} ) = e^{w_a \lv} \circ I_a(X^{\xv} ) \text{ for } \lv \in \Lv_0. \ee  In a slight abuse of notation, we define a map $w_a \star: V_S \rr V_S$ essentially as in \eqref{cg-act}: for $\lv \in \Lv$ set 
\be{cg:act-2} 
\begin{split}
w_a \star X^{\lv} := \frac{e^{- \la \lv, a \ra \av }}{1 - v e^{-m(\av) \av}} & \left[ (1-q^{-1}) e^{ \resi_{\m(a)} \left( \frac{ B(\lv, \av)}{\Q(\av)} \right) \av }  \right. \\
 & \hskip .3cm \left. - q^{-1} \g_{ \Q(\av) + \B(\lv, \av) } e^{ (\m(\av) - 1) \av} (1 - e^{-\m(\av) \av}) \right] \circ X^{\lv}
\end{split}
\ee
and extend to $V_S$ as in \S \ref{subs:CGaction}. To verify that $w_a \star -: \ring \rr \ring$ as defined in \eqref{cg:act-2} satisfy the braid relations under the specialization \eqref{par:choices}, it suffices to verify the same statement for the operators just introduced on $V_S.$ To show the latter, first note that the previous Lemma \ref{I_a:spw} can be rephrased as follows: as functions on the big cell, 
\be{Ia:cg} I_a(X^{\xv}) = \c(\av) \circ (w_a \star X^{\xv}) = \c(\av) (w_a \star X^{\xv}), \ee where the second equality uses the fact  (see \S \ref{B:Q}(iii)) that each $\m(\av) \av \in \Lv_0.$ Now if $w \in W$ admits a reduced decomposition $w =  w_{b_1} \cdots w_{b_k}$ with each $b_i \in \Pi$ one shows 
\be{I_w} I_{b_1} \cdots I_{b_k} (f) = \int_{U_w} f(w u_w x) du_w , \text{ for } f \in M(\tG) \ee wherever the integrals converge (notation: $U_w:= \prod_{a > 0, w^{-1} a < 0} U_a$, and $du_w$ is the product Haar measure on this group). As the right-hand side of \eqref{I_w} is independent of the choice of reduced decomposition, the same is true of the left-hand side. Hence, $I_{b_1} \cdots I_{b_k} (X^{\xv})$ (which lies in $V_S$) is independent of the choice of the reduced decomposition and is equal to a quantity which we denote by $I_w(X^{\xv})$. On the other hand, repeated application of \eqref{Ia:cg} and \eqref{Ia:act} shows \be{Iw:cg} \begin{array}{lcr} I_w(X^{\xv}) =  \c_w \circ ( w_{b_1} \star \cdots \star w_{b_r} \star X^{\xv})& \text{ where } & \c_w:= \prod_{\substack{\gv > 0 \\ w^{-1} \gv < 0}} \c(\gv) \end{array} \ee As the left-hand side and the factor $\c_w$ appearing on the right-hand side do not depend on the choice of reduced decomposition, it follows that $w_{b_1} \star \cdots \star w_{b_r} \star X^{\xv}$ also does not depend on the reduced decomposition, which is what we were aiming to show.

\tpoint{Averaging Operators}\footnote{The material in the remainder of this section is almost identical to the corresponding sections of \cite{pat-whit}} Let $\psi: \um \rr \C^*$ now be an unramified, principal character as in \S \ref{psi-sec}. An element $f: \tG \rr \C$ in $M(\tG)$ lies in the subspace $M(\tG, \um, \psi) \subset M(\tG)$ if it additionally satisfies \be{um-psi-right} f(g n^-) = \psi(n^-) f(g) \text{ for } g \in \tG, n^- \in U^-. \ee For each $\lv \in \Lv_+,$ define the averaging operator $\Av_{\psi, \lv}: M(\tG, \tI^-)  \rr M(\tG, \um, \psi)$ as follows:  for $x \in \aw$ and $\tve_x$ as in \S \ref{vx} we set  \be{avg:2} \Av_{\psi, \lv} (\tve_x) (g) = \sum_{y \in m_{x, \lv}^{-1} (g) } \iota(\z(y)) \, \psi(\niw(y)) \text{ for } g \in \tG,\ee where $m_{x, \lv}: \bmu \tA_{\O} U \dot{x} \tI^- \times_{\tI^-} \tI^- \pi^{\lv} \um \rr \tG$ is the usual multiplication map and $\niw, \z$ were defined as in \eqref{proj-fib}. It is easy to see that $\Av_{\psi, \lv}(\tve_x) \in M(\tG, U^-, \psi)$; denote by $\Av^{\gen}_{\psi, \lv}$ its restriction to the big cell, i.e., \be{av-def} \Av^{\gen}_{\psi, \lv}(\tve_x)(g) = \begin{cases} \Av_{\psi, \lv}(\tve_x)(g) & \text{ if } g \in \bmu \tB \um \\
0 & \text{ otherwise} . \end{cases} \ee We may extend the above construction linearly to define maps \be{av-maps} \begin{array}{lcr} \Av_{\psi, \lv}: M(\tG, \tI^-)  \rr M(\tG, \um, \psi) \text{ and } \Av^{\gen}_{\psi, \lv}: M(\tG, \tI^-)  \rr M^{\gen}(\tG, \um, \psi) \end{array}. \ee For any $f \in M(G, \tI^-)$ for which the formula makes sense (i.e. is convergent in some sense), one has  \be{Ia-av} I_a (\Av_{\psi, \lv}(f)) (g) = \Av_{\psi, \lv}(I_a(f))(g) \text{ for } g \in \tG. \ee This is essentially an (algebraic) version of the Fubini theorem.

\spoint For $w \in W,$ note that $\Av_{\psi, \lv}(\tve_w): \tG \rr \C^*$ is $\tA_{\O}U$-left invariant and so 
\be{av-1hs} \Av^{\gen}_{\psi, \lv}(\tve_w) = \sum_{\xv \in \Lv} c_{\xv} \hs_{\psi, \xv}. \ee The following result is proven easily from the definitions (see also \cite[Lemma 5.4]{pat-whit}).
 
\begin{nlem} \label{whit:as-av} For $w \in W,$ $\lv \in \Lv_+,$ we have \be{avg:whit:1} \W_{w, \lv} =  \sum_{\mv \in \Lv} e^{\mv} q^{ \la \rho, \mv \ra} \Av_{\psi, \lv}(\tve_w)(\pi^{\mv}). \ee  \end{nlem} Note that if we write $\W_{w, \lv} = \sum_{\mv \in \Lv} a_{\mv} e^{\mv},$ then $a_{\mv} = q^{ \la \rho, \mv \ra} c_{\mv},$ where $c_{\mv}$ are as in \eqref{av-1hs}. 

\tpoint{Notation} \label{Psi-not} In the sequel, we adopt the following notation: for a function $f \in \C(\tA_{\O} \up \setminus G)$ we define the formal series $\Psi(f)$ which is constructed from the values of $f$ on the torus as follows,  \be{Psi:f} \Psi(f):= \sum_{\mv \in \Lv} f(\pi^{\mv}) e^{\mv} q^{\la \rho, \mv \ra }. \ee It is easy to see that \be{Psi:e} \Psi (e^{\mv} f) = e^{\mv} \Psi(f) \text{ for } \mv \in \Lv_0 \ee where in the left-hand side $e^{\mv} f$ is the function defined as in \eqref{act:f} and in the right-hand side, $e^{\mv}$ acts by multiplication on the formal series $\Psi(f).$ Thus Lemma \ref{whit:as-av} states  
\be{Psi-whit} \Psi(\Av_{\psi, \lv}(\tve_w) ) = \W_{w, \lv} \text{ for } w \in W. \ee

\section{Proof of the Theorem  \ref{main-rec} } \label{rec-pf}

In this section, we finish the proof of Theorem \ref{main-rec}. 

\spoint Let $w \in W$ with $w=w_a w'$ with $a \in \Pi$ and $\ell(w) =\ell(w')+1.$ We wish to show that $\T_a( \W_{w', \lv} ) = \W_{w, \lv}.$ For $\tauv \in \Lv$, define $a_{\tauv} \in \C$ by the expression \be{w':a} \W_{w', \lv} = \sum_{\tauv \in \Lv} a_{\tauv} e^{\tauv}. \ee Using the expression for $\T_a$ from \eqref{T:c-b}, Theorem \ref{main-rec} amounts to showing  \be{W:as} \W_{w, \lv} = \sum_{\tauv \in \Lv} a_{\tauv} \T_a(e^{\tauv}) = \b(\av) \sum_{\tauv \in \Lv}  a_{\tauv} e^{\tauv} + \c(\av) \sum_{\tauv \in \Lv}  a_{\tauv} w_a \star e^{\tauv} \ee where the rational functions $\b(\av)$ and $\c(\av)$ defined in \eqref{c:b} are now taken to be expansions in positive powers of $e^{\m(\av) \av}.$

\spoint Recall from \eqref{Ia-av} that we have an equality of functions on $\tG$, \be{rec:1} I_a ( \Av_{\psi, \lv}(\tve_{w'}) ) =  \Av_{\psi, \lv}(I_a (\tve_{w'} )). \ee The result \eqref{W:as} will follow by comparing both sides of this formula.  We begin with the right-hand side (r.h.s.) of \eqref{rec:1}. From Lemma \ref{lem-int-vw},  $I_a (\tve_{w'}) = \tve_{w_a w'} - \b(\av) \tve_{w'},$ and so the right-hand side of \eqref{rec:1} is equal to \be{rec:2} \Av_{\psi, \lv}(I_a (\tve_{w'} )) = \Av_{\psi, \lv}(\tve_{w_a w'} ) -  \b(\av) \Av_{\psi, \lv}(\tve_{w'}). \ee Applying $\Psi$ to \eqref{rec:2} we obtain using \eqref{Psi-whit} \be{p-rec:0} \Psi( \text{ r.h.s. of \eqref{rec:1} } ) =  \W_{w, \lv} - \b(\av) \W_{w', \lv}  . \ee 

\spoint Now we turn to the left-hand side of \eqref{rec:1}. We begin by noting (see \cite[Claim 5.6]{pat-whit} for a proof) \be{i-gen:0} I_a ( \Av_{\psi, \lv}(\tve_{w'}) ) (g)  = I_a ( \Av^{\gen}_{\psi, \lv}(\tve_{w'}) ) (g) \text{ for } g \in \tB \um.\ee  By definition we can write $\Av^{\gen}_{\psi, \lv}(\tve_{w'}) = \sum_{ \xv \in \Lv} c_{\xv}  \hs_{\psi, \xv}$ with  $c_{\xv} = \Av^{\gen}_{\psi, \lv}(\tve_{w'})(\pi^{\xv}).$ Using \eqref{w':a} and the fact that $\Psi(\Av^{\gen}_{\psi, \lv}(\tve_{w'}))$ is equal to $ \W_{w', \lv}$ , we conclude that \be{a:c-1}  a_{\xv} = c_{\xv} q^{ \la \rho, \xv \ra} \text{ for } \xv \in \Lv .\ee For $g$ in the big cell we have \be{rec:4} I_a ( \Av_{\psi, \lv}(\tve_{w'}) ) (g) = I_a ( \Av^{\gen}_{\psi, \lv}(\tve_{w'}) ) (g) =   \sum_{\xv \in \Lv} c_{\xv} \Gamma_{a, \xv}(g). \ee Thus applying $\Psi,$ we find  \be{p-rec:1} \Psi(\text{ l.h.s. of \eqref{rec:1} }) = \sum_{\mv \in \Lv} e^{\mv} q^{ \la \rho, \mv \ra} \left( \sum_{\xv \in \Lv} c_{\xv} \Gamma_{a, \xv}(\pi^{\mv}) \right ). \ee Recalling the expression for $\Gamma_{a, \xv}=\Gamma^I_{a, \xv} + \Gamma^{II}_{a, \xv} $ from \eqref{Gam1}, \eqref{Gam2}, we find
\be{p-rec:2-1} \sum_{\xv \in \Lv} c_{\xv} \Gamma^I_{a, \xv}(\pi^{\mv}) &=& q^{-1} \, \sum_{\substack{\xv \in \Lv \\ \mv = w_a \xv - \av  }} c_{\xv} q^{ 1 + \la \xv, a \ra } \g_{ \Q(\av) (\la \xv, a \ra +1 )}  \\  \sum_{\xv \in \Lv} c_{\xv} \Gamma^{II}_{a, \xv}(\pi^{\mv}) &=& \sum_{\substack{\xv \in \Lv \\ \mv = w_a \xv + m \av \\  m \geq 0, \, m \equiv \la \xv, a \ra \mod{\m(\av)}   }} (1 - q^{-1}) c_{\xv} q^{ -m + \la \xv, a \ra }. \ee Hence the right-hand side of \eqref{p-rec:1} may be written as 
 \be{p-rec:3}&&  \sum_{\xv \in \Lv} e^{w_a \xv - \av} q^{ \la \rho, w_a \xv - \av \ra } c_{\xv} q^{ 1 + \la \xv, a \ra }q^{-1} \g_{ \Q(\av) (\la \xv, a \ra +1 )}  \\ \label{p-rec:4} &+&   \sum_{\substack{ \xv \in \Lv, m \geq 0 \\ m \equiv \la \xv, a \ra \mod{\m(\av)}  }} e^{ w_a \xv + m \av} q^{ \la \rho, w_a \xv + m \av \ra} (1 - q^{-1}) c_{\xv} q^{ -m  + \la \xv, a \ra } \ee Using \eqref{a:c-1}, one finds that \eqref{p-rec:3} is equal to \be{p-rec:3b} \sum_{\xv \in \Lv} e^{w_a \xv - \av} q^{ \la \rho, \xv \ra} c_{\xv} q^{-1} \g_{ \Q(\av) (\la \xv, a \ra +1 )}  = \sum_{\xv \in \Lv} a_{\xv} e^{w_a \xv - \av} q^{-1} \g_{ \Q(\av) (\la \xv, a \ra +1 )} . \ee Similarly, we find that \eqref{p-rec:4} is equal to \be{p-rec:4b}  \sum_{\substack{\xv \in \Lv, m \geq 0 \\ m \equiv \la \xv, a \ra \mod{\m(\av)} }} e^{ w_a \xv + m \av} q^{ \la \rho,  \xv \ra} (1 - q^{-1}) c_{\xv}  =  \frac{1 - q^{-1}}{1 - e^{\m(\av) \av}} \sum_{\xv \in \Lv} a_{\xv} e^{w_a \xv}e^{( \resi_{\m(\av)} \la \xv, a \ra)  \av} .  \ee Using \eqref{c:act}, we summarize our results and conclude the proof,  \be{p-rec:5} \eqref{p-rec:1} = \sum_{\xv \in \Lv} \c(\av) a_{\xv} w_a \star e^{\xv}. \ee

\renewcommand{\sph}{\widetilde{\mathbf{1}}_K}
\newcommand{\Wh}{\mathsf{W}}
\newcommand{\eval}{\mathsf{ev}}
\tpoint{Connections to representation theory } \label{s:whit-fun} Let us now make some remarks that place our constructions in a more representation theoretical framework. Recall that we have defined the space $M(\tG)$ in \S \ref{prin-ser}- \ref{prin-ser:2}. We have been vague about the notion of support for elements in $M(\tG),$ but if we impose the condition of compact support, then the resulting function space becomes the ``universal'' principal series representation, whose elements are functions (with some prescribed support condition) $f: \tG \rr \C[\Lv]$ satisfying 
\be{prin:ind} f ( \zeta a_{\O} \pi^{\mv} u g)  = q^{ -\la  \mv ,\rho\ra}  \iota(\zeta)  (e^{- \mv} f)(g) \text { for } \zeta \in \bmu, a_{\O} \in \tA_{\O}, \mv \in \Lv_0, u \in U,  g \in \tG.  \ee 
To an element $\varphi \in M(\tG)$ we define an element in the (universal) principal series 
\be{map:real} \varphi \mapsto f_{\varphi}(g):= \sum_{\mv \in \Lv} \varphi(\pi^{\mv} g) e^{\mv} q^{ \la \mv ,\rho \ra}. \ee  
It follows from \eqref{prin:ind} that every element of the (universal) principal series is equal to $f_{\varphi}$ for some $\varphi \in M(\tG).$ 
Let $\sph$ be the function on $\tG$ defined as follows 
\be{def:sph,iw} \sph(\zeta a_{\O} u \zeta  \pi^{\mv} k) &=& \begin{cases} 0 & \text{ if } \mv \neq 0 \\ \iota(\zeta) & \text{ if } \mv = 0 \end{cases} \ee 
where $\zeta \in \bmu, u \in U, a_{\O} \in \tA_{\O}, k \in K$ and $\zeta \in \bmu.$ Recall also that we have defined the function $\tve_w \, (w \in W)$ in \eqref{vew}; one sees that $\sph = \sum_{w \in W} \tve_w$. Using this notation, set 
\be{sph-vects} \begin{array}{lcr} f_K:= f_{\sph}& \text{ and } & f_w:= f_{\tve_w}. \end{array} \ee
The element $f_K$ is often called a \emph{spherical vector} in the principal series representation. 
 
 For a fixed character $\psi$ as above, a Whittaker \emph{functional} will be a map $L: M(\tG) \rr \C$ satisfying \be{whit-fun} L(n^-.\varphi) = L(\varphi) \psi(n^-) \text{ for } n^- \in U^-,\ \varphi \in M(\tG). \ee From such a functional and a choice of $\varphi$ we obtain a corresponding function $L(\varphi)(g):= L(g. \varphi).$ Note the difference with the transformation properties of our functions $\mc{W}$ on $\tG$ as described in \eqref{W-inv}-- we will see below the precise relation between these two functions.
 
  For each $\xv \in \Lv$ we construct a functional via the following integral 
  \be{L:xv} L_{\xv}( \varphi ) = q^{\la \xv,\rho\ra}\int_{U^-} \varphi( \pi^{\xv} u^- ) \psi(u^-)^{-1} du^- \text{ for } \varphi \in M(\tG), \ee 
  where $du^-$ is a suitably normalized Haar measure on $U^-.$ Note that there exists an action of $\C[\Lv_0]$ (and actually some localizations of this ring) on the Whittaker functionals: 
  \be{act-on-fun} (e^{\mv} L_{\xv}) ( \varphi ) = L_{\xv} ( e^{-\mv} \varphi) = L_{\mv + \xv} (\varphi) \text{ for } \mv \in \Lv_0. \ee 
For $\varphi \in M(\tG)$ and $f_{\varphi}$ as in \eqref{map:real} we define 
\be{L-all} \mc{L}(f_{\varphi})(g) &=& \int_{U^-} f_{\varphi}(u^- g ) \psi(u^-)^{-1} du^- 
  = \sum_{\mv \in \Lv} e^{\mv} q^{ \la \mv, \rho \ra} \int_{U^-} \varphi(\pi^{\mv} u^- g ) \psi(u^-)^{-1} \\
  &=& \sum_{\mv \in \Lv}  L_{\mv}(\varphi)(g) e^{\mv} .\ee  
  
   In this language, one defines the \emph{spherical Whittaker function} as $\mc{L}(f_K)(g)$ 
   and the Iwahori-Whittaker functions as $\mc{L}(f_w)(g).$ 
   The usual argument shows that $\mc{L}(f_K)(g)$ 
   is determined by its values at $g=\pi^{-\lv} \, (\lv \in \Lv_+).$ Note again that the switch between dominant in Proposition \ref{p:W-inv} and antidominant here.

   Finally we want to explain the connection between  $\mc{L}(f_K)(\pi^{-\lv})$ 
   and $\mc{L}(f_w)(\pi^{-\lv})$. First, for a suitable choice of Haar  measure on $U^-$ we can arrange that 
   \be{avg-L} \Av_{\psi, \lv}(\varphi)(\pi^{\mv}) =q^{ -\la \mv ,\rho \ra} L_{\mv} ({\varphi}) (\pi^{-\lv}) \text{ for } \varphi \in M(\tG, \tI^-). \ee 
   Note incidentally that our definition of the averaging operator in \eqref{avg:2} was only for elements from $M(\tG, \tI^-),$ but we can use the above to extend it all of $M(\tG)$ (though we do not need this here).  From \eqref{Psi-whit}, we have $\W_{w, \lv} = \Psi(\Av_{\psi, \lv}(\tve_w)),$ and so using \eqref{avg-L} above, we find  
   \be{W:w-L} \W_{w, \lv} &=&   \sum_{\mv} L_{\mv}(\tve_w)(\pi^{-\lv}) e^{\mv}  = \mc{L}(f_w)(\pi^{-\lv}). \ee 

\tpoint{Connection to the work of Kazhdan-Patterson} \label{s:kaz-pat} In \cite{ka:pat}, Kazhdan and Patterson computed the effect of composing the intertwiner $I_w \, (w \in W)$ defined in \eqref{I_w} with the Whittaker functionals $L_{\mv}$ introduced in the previous paragraph. They expressed their answer as an equality of functionals on $M(\tG)$ \be{tau} L_{\mv} \circ I_w = \sum_{\xv \in \Lv} \tau^w({\mv, \xv}) \, L_{\xv} \ee where $\tau^w(\mv, \xv)$ is some rational function in the variables $e^{\lv} \, (\lv \in \Lv_0)$ and the action of $e^{\lv}$ on the functional is as in \eqref{act-on-fun}. If $a \in \Pi$ the formula above takes on the simple form (see \cite[Theorem 13.1 and p.26]{mac-met}\footnote{Note that one must take $\av$ to $-\av$ as well as renormalize the intertwining operators from the formulas of \cite{mac-met} to match the conventions of this paper, see Remark \ref{rmk:matchupwPMcN}}), \be{KP-sim} L_{\mv} \circ I_a = \tau^{w_a}(\mv, \mv) L_{\mv} + \tau^{w_a}(\mv, w_a \mv - \av) L_{w_a \mv - \av} \ee for some explicit rational functions $\tau^{w_a}(\mv, \mv)$ and $\tau^{w_a}(\mv, w_a \mv - \av).$

 Let us see how from the proof of Theorem \ref{main-rec}, we obtain \eqref{KP-sim} in the special case when this equality of functionals is applied to $\tve_w$ and then evaluated at $\pi^{-\lv} \, (\lv \in \Lv_+).$ In the process, we obtain a formula for the functions $\tau^{w_a}(\mv, w_a \mv - \av)$ and $\tau^{w_a}(\mv, \mv).$

\emph{Notation:} For any series $F = \sum_{\xv} c_{\xv} e^{\xv},$ let us write $[e^{\mv}] F = c_{\mv}.$

\noindent If we write $\Psi(\Av_{\psi, \lv}(\tve_w)) = \sum_{\xv} a_{\xv} e^{\xv}$, then from (\ref{W:w-L}) we have \be{L-coeff} [e^{\mv}]  \Psi(\Av_{\psi, \lv}(\tve_w)) = a_{\mv} = L_{\mv}( \tve_w) (\pi^{-\lv})  .\ee  From the proof of Theorem \ref{main-rec} (see \eqref{w':a}, \eqref{rec:1}, \eqref{p-rec:5} and \eqref{Psi-whit}), we obtain \be{p-rec:5-1} \Psi( \Av_{\psi, \lv}(I_a(\tve_w))) = \c(\av) w_a \star \Psi(\Av_{\psi, \lv}(\tve_w)).\ee  Using \eqref{avg-L}, we obtain \be{Av: L} [e^{\mv}] \Psi( \Av_{\lv, \psi}(I_a(\tve_w))) = L_{\mv}(I_a(\tve_w))(\pi^{-\lv}).\ee To connect with \eqref{KP-sim} we now need to analyze $[e^{\mv}] \c(\av) w_a \star \Psi(\Av_{\psi, \lv}(\tve_w)).$ Using the formula \eqref{c:act}, this amounts to studying \be{ots:1}  \underbrace{[e^{\mv}] \sum_{\xv} a_{\xv} \frac{1-q^{-1}}{1 - e^{\m(\av) \av}} e^{w_a \xv} \, e^{( \resi_{\m(\av)} \la \xv, a \ra)  \av} }_{(I)} \ \ + \ \   \underbrace{ [e^{\mv}] \sum_{\xv} a_{\xv}  \,  q^{-1} \gf_{(1 + \la \xv, a \ra) \Q(\av)} e^{w_a \xv - \av}}_{(II)} . \ee 

\begin{description}
 
%
\item[(I)] We reparametrize this sum by the involution $\xv \mapsto w_a \xv + (\resi_{\m(\av)} \la \xv, a \ra) \av$. Then 
\be{I:re} (I) = [e^{\mv}] \sum_{\xv} a_{w_a \xv + (\resi_{\m(\av)} \la \xv, a \ra) \av } \frac{1-q^{-1}}{1 - e^{\m(\av) \av}} e^{\xv}. \ee 
If we write $w_a \xv +(\resi_{\m(\av)} \la \xv, a \ra) \av = \xv -  \la \xv, a \ra + ( \resi_{\m(\av)} \la \xv, a \ra  ) \av,$ and notice 
$$- \la \xv, a \ra \av + (\resi_{\m(\av)} \la \xv, a \ra)  \av \in \Lv_0,$$ we see that (I) is equal to \be{I:re2} & & [e^{\mv}]  \frac{1-q^{-1}}{1 - e^{\m(\av) \av}} \sum_{\xv} a_{\xv - \la \xv, a \ra \av + \left( \resi_{\m(\av)} \la \xv, a \ra \right) \av } e^{\xv} \\ &=&  \frac{1-q^{-1}}{1 - e^{- \m(\av) \av}} e^{- \la \mv, a \ra \av + \left(\resi_{\m(\av)} \la \mv, a \ra \right) \av}  L_{\mv} (\tve_w) (\pi^{-\lv}) \ee  



\item[(II)] Reindexing $(II)$ using the map $\xv \mapsto w_a \xv - \av,$ we find 
 \be{II;1}  (II) =  a_{w_a \mv - \av }  \,  q^{-1} \gf_{(-1 - \la \mv , a \ra) \Q(\av)} = q^{-1} \gf_{(-\Q(\av) - \B( \mv , \av)) } L_{w_a \mv - \av} (\tve_w) (\pi^{-\lv}).  
\ee

%
\end{description} 

In sum, if we define  
\be{C'D'} \begin{array}{lcl} \tau^{w_a}(\mv, \mv) :=  \frac{1-q^{-1}}{1 - e^{-\m(\av) \av}} e^{- \la \mv, a \ra \av +( \resi_{\m(\av)} \la \mv, a \ra)  \av}  & \text{ and } \\  
\tau^{w_a}(\mv, w_a \mv - \av) := q^{-1} \gf_{(-\Q(\av) - \B( \mv , \av) )}, \end{array} \ee 
we obtain (cf. \ref{KP-sim} or \cite[p.26]{mac-met}) that  \be{fin:kp} L_{\mv} ( I_a (\tve_w) ) (\pi^{-\lv}) = \tau^{w_a}(\mv, \mv) \, L_{\mv} (\tve_w)(\pi^{-\lv}) + \tau^{w_a}(\mv, w_a \mv - \av) \, L_{w_a \mv - \av} (\tve_w)(\pi^{-\lv}). \ee

\begin{nrem}\label{rmk:matchupwPMcN}
To compare \eqref{C'D'} with the explicit formulas for $\tau_{a,b}^1$ and $\tau_{a,b}^2$ in \cite[Theorem 13.1]{mac-met}, note that for any integer $k$ and $\m=\m(\av)$ we have 
$k+\resi _{\m}(-k) = \left\lceil \frac{k}{\m}\right\rceil\cdot \m.$
Then take $\av$ to $-\av$ on the right-hand side of both equations in \eqref{C'D'} to get (using here \eqref{w-inv})
\be{C'D'var} \tau^{(1)}= \frac{1-q^{-1}}{1 - e^{\m(\av) \av}} e^{ \left\lceil \frac{\B( \mv , \av)}{\m(\av) \Q(\av)}\right\rceil \m(\av ) \av}   \text{ and } \tau^{(2)}= q^{-1} \gf_{(-\Q(\av) + \B( \mv , \av) )}. \ee 
Now writing $x_{\alpha }^{n_{\alpha }}$ for $e^{\m(\av)\av}$ and multiplying by $\frac{1-x_{\alpha }^{n_{\alpha }}}{1-q^{-1}x_{\alpha }^{n_{\alpha }}}$ gives the same formulas as in \emph{loc. cit}.
\end{nrem}

%

\appendix

\section{Spherical Functions} \label{app-sph}

\newcommand{\s}{\mathbb{S}}
\newcommand{\Js}{\mathbb{J}}
\newcommand{\Ts}{\Ts}
\newcommand{\mk}{m_{\lv, \tK}}
\newcommand{\mum}{m_{\lv, U^-}}

The aim of this appendix is to sketch a proof of the computation of the metaplectic spherical function.  We only provide an outline of the argument here; the details are very similar to \cite[\S 7]{bkp}. In a forthcoming paper the affine, metaplectic situation will be treated in more detail.

\spoint \label{note:app} Let $v$ be a formal parameter and set $\C_v=\C[v, v^{-1}].$ We consider the group algebra over this ring $\C_v[\Lv],$ which admits a map $\C_v[\Lv] \rr \C[\Lv]$ by sending $v \mapsto q^{-1},$ where $q$ is the size of our residue field. We call the image of an element under this map its specialization at $v=q^{-1}.$ The natural action of $W$ on $\Lv$ induces an action of $W$ on $\C_v(\Lv)$ (the ring of rational functions on $\Lv$), which we denote by $f \mapsto f^w$ where $f \in \C_v(\Lv), w \in W$. Let $\C(\Lv)[W]$ be the twisted group algebra over this ring: i.e. if $f, g \in \C_v(\Lv)$ and $w, w' \in W$ then $f [w] g [w'] = f g^w [w w'].$

\spoint For any $\lv \in \Lv_+$ let \be{m:lv} m_{\lv, \tK}: \bmu \tA_{\O} U \tK \times_{\tK} \tK \pi^{\lv} \tK \rr \tG \ee be the map induced from multiplication. For each $\mv \in \Lv$ we may define a map $\z: m_{\lv, K}^{-1}(\pi^{\mv}) \rr \bmu$ by projection onto the $\bmu$ component from the first factor (one can check the process is well-defined since $\tK \cap \bmu = \tK \cap \tA \cap \bmu =1$). We then define the spherical function\footnote{More precisely, this is the image of the Satake transform of the characteristic function $K \pi^{\lv} K$  for metaplectic groups, but we do not develop the framework to make this statement precise here.}   (which can be shown to lie in $\C[\Lv]$),
  \be{spy:def} \s(\pi^{\lv}) := q^{- \la 2 \rho, \lv \ra} \sum_{\mv \in \Lv} e^{\mv} q^{ \la \rho, \lv \ra}\sum_{ x \in m_{\lv, \tK}^{-1}(\pi^{\mv} )}  \iota(\z(x)) , \ee where we recall that we fixed an embedding $\iota: \bmu \hookrightarrow \C^*.$ 
  
\begin{nrems} Starting instead with the map $\mum: \bmu \tA_{\O} U \tK \times_{\tK} \tK \pi^{\lv} U^- \rr \tG,$ we may also consider projections \be{projs} \begin{array}{lcr} \z: \mum^{-1}(\pi^{\mv}) \rr \bmu & \text{ and } & \nn: \mum^{-1}(\pi^{\mv}) \rr U^-_{\O} \setminus U^-. \end{array} \ee In this terminology this Whittaker function (introduced earlier in \eqref{W-cov}; see also \eqref{W:dec}) is written as \be{Whit:m} \W(\pi^{\lv}) = q^{ - \la 2 \rho, \lv \ra} \sum_{\mv \in \Lv} e^{\mv} q^{ \la \rho, \mv \ra} \sum_{ x \in \mum^{-1}(\pi^{\mv})} \iota(\z(x)) \psi(\nn(x)). \ee

\end{nrems}  
  
\spoint  To state our main result concerning $\s(\pi^{\lv})$ we introduce some further notation.  For each $\lv \in \Lv_+$ let $W_{\lv} \subset W$ denote its stabilizer; it is also a Coxeter group generated by the simple reflections it contains.  For each subset $\Sigma \subset W$ we define its Poincare polynomial, \be{Sig-poin} \Sigma(v^{-1}) := \sum_{w \in \Sigma} v^{- \ell(w)}. \ee  Recall also the definition of $\Delta_v$ and $\Delta$ from \eqref{del:v}; for simplicity we sometimes write \be{Gam} \Gamma:= \frac{\Delta_v}{\Delta}= \prod_{a \in R_+} \frac{1 - v e^{ - \m(\av) \av}}{1-e^{- \m(\av)\av} } , \ee and also denote, in an abuse of notation, its specialization at $v=q^{-1}$ by the same notation.

\begin{nthm} \label{sph-fla} For each $\lv \in \Lv_+$ the expression $\frac{v^{ \la \rho, \lv \ra} }{W_{\lv}(v)} \sum_{w \in W} \Gamma^w e^{w \lv} \in \C_v[\Lv].$ Its specialization at $v=q^{-1}$ is equal to $\s(\pi^{\lv})$, i.e. we may write \be{s:spec} \s(\pi^{\lv}) = \frac{q^{ - \la \rho ,\lv \ra }}{W_{\lv}(q^{-1})} \sum_{w \in W} \Gamma^w e^{w \lv}, \ee  
\end{nthm}

To prove the Theorem, we follow the same pattern as in our computation of the Whittaker function, which in turn was based on \cite[\S7.2-7.3]{bkp}: first, we decompose $\s(\pi^{\lv})$ into a sum over Iwahori-pieces; then, we show that each Iwahori piece admits a description in terms of certain Demazure-Lusztig type operators; and finally, we conclude the theorem by applying a combinatorial identity for the sum of such Demazure-Lusztig operators. 

\tpoint{Remark} \label{zero-case-rmk} If $\lv=0$ then it is easy to see from the definition \eqref{spy:def} that $\s(\pi^0)=1.$ Hence, from Theorem \ref{sph-fla}, the right-hand side of \eqref{s:spec} must be equal to $1$ (at $v=q^{-1}$). This in turn reduces to a combinatorial identity  \be{zero-case} W(q^{-1}) = \sum_{w \in W} \Gamma^w \ee which was proven independently by Macdonald \cite[Theorem 2.8]{mac:poin}. In fact \eqref{zero-case} holds for any finite Coxeter group (e.g. $W_{\lv}$). On the other hand, the naive analogue of \eqref{zero-case} is known to fail for affine root systems (see \cite{mac:for}), and its failure is precisley captured by a factor stemming from 'constant term conjecture' of Macdonald (Cherednik's Theorem). 

\newcommand{\cs}{\mathbb{c}}
\newcommand{\bs}{\mathbb{b}}
\renewcommand{\Ts}{\mathbb{T}}
\renewcommand{\Ps}{\mathbb{P}}
\newcommand{\ms}{\mathbb{m}}

\spoint We begin by introducing the precise Demazure-Lusztig operators which will be needed in the sequel. For each $a \in \Pi,$ we define elements in $\C_v(\Lv)$  
\be{c:b_sph} \begin{array}{lcr} \cs(\av):= \frac{1 - v e^{ \m(\av) \av}}{1-e^{ \m(\av) \av}} & \text{ and } & \bs(\av):= \frac{v - 1}{1- e^{ \m(\av) \av}} \end{array}. \ee  Using these, we define the 'spherical' Demazure-Lusztig operators as the following elements in $\C_v(\Lv)[W]$   \be{T-sph} \Ts_a =  \cs(\av) [w_a] +  \bs(\av) [1]  \ee One has the following facts,
\begin{enumerate}
\item The $\Ts_a$ satisfy the braid relations; this follows from \cite[(8.3)]{lus-K}. Hence we may define elements $\Ts_w$ unambiguously using any reduced decomposition for $w.$

\item The $\Ts_a$ satisfy the following Hecke relations \be{hec-rel} (\Ts_a+1)(\Ts_a - v) =0 \text{ for } a \in \Pi, \ee as one may check by a direct calculation.

\item Let $\Ps:= \sum_{w \in W} \Ts_w,$ the spherical symmetrizer. Then we have the identity in $\C(\Lv)[W]$ \be{mac-for} \Ps:= \ms \sum_{w \in W} \Gamma^w [w],\ee where $\ms \in \C(\Lv)$ is some $W$-invariant factor. This follows as in \cite[Proposition 7.3.12]{bkp} (the argument there is essentially due to Cherednik), and the the proof only uses the previous two facts. In particular, it does not use the existence of a long word in $W.$ If one uses the long word, one can immediately verify that ${\mathbb{m}}=1$ (the fact that ${\mathbb{m}}=1$ can also be deduced using the identity \eqref{zero-case} as we indicate at the end of \S \ref{conc}  below).

\item Using the standard action of $W$ on $\Lv$ we obtain an action of $\Ts_a: \C_v(\Lv) \rr \C_v(\Lv)$. In fact, $\Ts_a: \C_v[\Lv] \rr \C_v[\Lv].$ Moreover, if $\lv \in \Lv_+$ and $w \in W_{\lv}$ one verifies that \be{T-stab} \Ts_w (e^{\lv}) = v^{\ell(w)} e^{\lv}. \ee 

\end{enumerate}

\newcommand{\mwl}{m_{w, \lv}}

\newcommand{\tP}{\tilde{P}}

\spoint For each $\lv \in \Lv_+,$ let us define \be{P:lv} \tP:= \sqcup_{w \in W_{\lv} } \tI^- w \tI^- .\ee We may consider the maps \be{mw} m_{w,\lv}: \bmu \tA_{\O} U w \tP \times_{\tP} \tP \pi^{\lv} \tK \rr \tG. \ee For each $\mv \in \Lv$ we may again define a projection \be{proj:sph-iw} \z_w: \mwl^{-1}(\pi^{\mv}) \rr \bmu. \ee The natural inclusion fits into a commutative diagram  \be{mw:mK} \xymatrix{\mwl^{-1}(\pi^{\mv}) \ar@{^{(}->}[rr] \ar[dr]^{\z_w} & & \ar[dl]_{\z} \mk^{-1}(\pi^{\mv})  \\ & \bmu & } \ee which is compatible with the projections $\z_w$ and $\z.$

\begin{nlem} \label{w-dec:sph} Let $\lv \in \Lv$ Let $W^{\lv}$ denote a fixed set of minimal length representatives for $W/W_{\lv}.$ For each $\mv \in \Lv$ there is a bijection (induced by the inclusions \eqref{mw:mK} above) \be{bij-fib} \sqcup_{w \in W^{\lv} } m_{w, \lv}^{-1}(\pi^{\mv}) \rr m_{\lv, \tK}^{-1}(\pi^{\mv}) \ee   \end{nlem}

The proof of this is similar to \cite[Lemma 7.3.3]{bkp}.

\spoint For each $w \in W^{\lv},$ we set 
 \be{Jw} \Js_w(\lv):= q^{ - \la 2 \rho, \lv \ra}  \sum_{\mv \in \Lv} e^{\mv} q^{ \la \rho, \mv \ra} \sum_{x \in m_{w, \lv}^{-1}(\pi^{\mv})} \iota(\z_w(x)). \ee From definition \eqref{spy:def} and the previous Lemma \ref{w-dec:sph}, we find that \be{s-J} \s(\pi^{\lv}) = \sum_{w \in W^{\lv}} \Js_w(\lv). \ee The same techniques as in the proof of \cite[Lemma 7.3.10]{bkp} can be adapted to show
 
\begin{nprop} \label{sph-rec} If $\lv \in \Lv_+, w \in W^{\lv},$  then $\Js_w(\lv)$ is the specialization of $v^{ \la \rho, \lv \ra} \Ts_w(e^{\lv})$ at $v=q^{-1},$ i.e.  \be{simp} \Js_w(\lv) = q^{- \la \rho, \lv \ra} \Ts_w(e^{\lv}) \ee \end{nprop} 

\begin{ncor} For any $\lv \in \Lv_+$ we have \label{sph:Ts} \be{sph-T} \s(\pi^{\lv}) = q^{- \la \rho, \lv \ra} \sum_{w \in W^{\lv}} \Ts_w(e^{\lv}).\ee \end{ncor}

\spoint Using \eqref{T-stab}, we find the following expression for $\Ps(e^{\lv}),$ \be{sph-T:1}  \Ps(e^{\lv}) = \sum_{w \in W} \Ts_w(e^{\lv}) &=& \sum_{w_1 \in W^{\lv}} \ \ \sum_{w_2 \in W_{\lv}} \Ts_{w_1 w_2}(e^{\lv})  \\
&=& \sum_{w_1 \in W^{\lv}} \ \ \sum_{w_2 \in W_{\lv}} \Ts_{w_1} v^{  \ell(w_1)} (e^{\lv}) \\
&=& W_{\lv}(v) \sum_{w \in W^{\lv}} \Ts_w(e^{\lv}). \ee Hence  \be{P:stab} \label{P:stab1} \sum_{w \in W^{\lv}} \Ts_w(e^{\lv}) &= &\frac{1}{W_{\lv}(v)} \Ps(e^{\lv}) =  \ms \frac{1}{W_{\lv}(v)} \sum_{w \in W} \Gamma^w e^{ w \lv}, \ee where the last equality follows from \eqref{mac-for} (note that $\ms$ here is independent of $\lv$).

\tpoint{Conclusion of Proof of Theorem \ref{sph-fla}} \label{conc} Using the above as well as \eqref{sph-T}, we conclude that the specialization  of \eqref{P:stab} at $v=q^{-1}$ are in $\C[\Lv].$ Moreover from \eqref{sph-T} we obtain \be{sph-T:2} \s(\pi^{\lv}) &=& \frac{q^{ -\la \rho, \lv \ra} }{W_{\lv} (q^{-1})} \sum_{w \in W} \Ts_w(e^{\lv}) \\ \label{sph-T:3} &=&  \ms \frac{q^{- \la \rho, \lv \ra} }{W_{\lv}(q^{-1})} \sum_{w \in W} \Gamma^w e^{w \lv}, \ee 
where $\ms$ is now the specialization of the element with the same name at $v=q^{-1}.$ 

The proof of Theorem \ref{sph-fla} will be concluded if we can show that $\ms=1.$ We may proceed in two ways. On the one hand, we could use the long word to conclude this in the expression \eqref{mac-for}.  Alternatively, if $\lv=0$ then the left-hand side of \eqref{sph-T:2} is equal to $1$ as we noted in Remark \ref{zero-case-rmk}. Thus  
\be{sph-T:4} \ms^{-1} = \frac{1}{W(q^{-1})} \sum_{w \in W} \Gamma^w .\ee Using Macdonald's identity \eqref{zero-case}, we can conclude that $\ms=1$ 

\tpoint{Remark} \label{whit:spy} Finally, we explain the relation between the factor $\ms$ and the factor $\mf{m}$ which arose in the combinatorial identity for the sum of metaplectic Demazure-Lusztig operators (see \ref{P:cs-op}).

\begin{enumerate}
\item If we apply an approximation result as in \cite[Proposition 4.3]{pat-whit} to the expression \eqref{Whit:m} and compare this with an approximation result as in \cite[Theorem 1.9(4), Theorem 1.13]{bgkp}, we conclude that \be{whit:sph-lim} \lim_{\lv \stackrel{+}{\rr} \infty} \frac{\W(\pi^{\lv})}{e^{\lv} q^{\la \rho, \lv \ra}} =  \lim_{\lv \stackrel{+}{\rr}\infty}  \frac{\s(\pi^{\lv})}{e^{\lv} q^{\la \rho, \lv \ra}}, \ee where $\lv \stackrel{+}{\rr} \infty$ indicates that $\lv$ is made increasingly dominant while remaining regular. In fact, both limits approach the Gindikin-Karpelevic integral.

%

\item Using the following expression (derived from (\ref{sym:P}, \ref{W:dec}, \ref{W:T-1}, \ref{P:cs-op}) \be{W:n} \W(\pi^{\lv}) = q^{ - \la  \rho, \lv \ra} \mf{m} \, \Gamma \, \sum_{w \in W} (-1)^{\ell(w)} \left( \prod_{\bv \in R^{\vee}(w^{-1}) } e^{-\m(\bv) \bv} \right) w \star e^{\lv}  , \ee one can show that as $\lv \stackrel{+}{\rr} \infty$ the left-hand side of \eqref{whit:sph-lim} is equal to $q^{ - \la  \rho, \lv  \ra } \mf{m} \, \Gamma.$ 

\item Applying an argument as in the proof of \cite[Theorem 1.13]{bgkp} to \eqref{sph-T:3} we find that the right-hand side of \eqref{whit:sph-lim} is equal to $q^{- \la \rho, \lv \ra } \ms \,  \Gamma.$ 

\item From the previous steps, we conclude that $\ms = \mf{m}.$ Using \eqref{sph-T:4} we then find that \be{whit-T:3} \mf{m}^{-1} =  \frac{1}{W(q^{-1})} \sum_{w \in W} \Gamma^w. \ee   In other words,  the $W$-invariant factor $\mf{m}$ in the Casselman-Shalika formula can be computed in terms of the corresponding factor in the  spherical function, which in turn can be computed by the expression on the right-hand side of \eqref{whit-T:3}.

\end{enumerate}

\section{On the braid relations for the Chinta-Gunnells action}\label{app:braid}

\spoint In this section, we prove some statements from \S \ref{w-defs}, concerning braid relations. These were used in \S \ref{sym-fla} and \S \ref{met-cs} to deduce the metaplectic Casselman-Shalika formula from our main result Theorem \ref{main-rec}. The notation is as in \S \ref{sec-met-dl}. 

\begin{nprop} \label{well-def} Let $w \in W$ and choose any reduced decomposition $w=w_{b_1} \cdots w_{b_r}$ with $b_i \in \Pi.$ For any fixed $f \in \nice$ and $g \in \C_v[\Lv]$ the elements \be{w-T-def} \begin{array}{lcr} w_{b_1} \star \cdots \star w_{b_r} \star f & \text{ and } & T_{w_{b_1}} \cdots T_{w_{b_r}} (g) \end{array} \ee are well-defined, independent of the choice of chosen reduced decomposition for $w.$ Hence they can be denoted unambiguously as $w \star f$ and $\T_w(g).$  \end{nprop}

 The remainder of this section is devoted to the proof of this statement and is based on three facts: (a) the group-theoretical result Theorem \ref{main-rec} and especially Remark \ref{rmk:DL-braid}; (b) a non-metaplectic analogue of the Proposition as stated in \S \ref{non-met}; and (c) a claim about rank two root systems. 
 

%


\spoint We begin with the following simple observation,

\begin{nlem} \label{weak->strong} Let $w \in W$ and choose any reduced decomposition $w=w_{b_1} \cdots w_{b_r}$ with $b_i \in \Pi.$ If \be{w:1} w_{b_1} \star \cdots \star w_{b_r} (e^{\mv}), \, \text{ for any } \mv \in \Lv_+ \ee is well defined, independent of the choice of reduced decomposition, then the same fact holds for any $\mv \in \Lv.$ \end{nlem}

\begin{proof} For $\mv \in \Lv,$ choose $\xv \in \Lv_0$ such that $\mv + \xv \in \Lv_+.$ As $\xv \in \Lv_0$, Lemma \ref{cg-basic-prop}(2) implies that \be{w-g:2} w_{b_1} \star \cdots \star w_{b_r} (e^{\xv + \mv}) &=& e^{w_{b_1} \cdots w_{b_r} \xv } w_{b_1} \star \cdots \star w_{b_r} (e^{\mv}) = e^{w \xv} w_{b_1} \star \cdots \star w_{b_r} (e^{\mv}) \ee The left-hand side of \eqref{w-g:2} is well-defined by assumption,  and the Lemma follows.
\end{proof}  

\newcommand{\ty}{\mathbf{y}}
\newcommand{\tw}{\mathbf{w}}
\newcommand{\tS}{S}

\spoint \label{braid-form} For the remainder of this section, we assume $R$ is of rank 2 with $\Pi^{\vee}= \{ \av, \bv \},$ $\Lv= \zee \av \oplus \zee \bv,$ and $W$ be the corresponding Weyl group generated by simple reflections $s, t$ through $a$ and $b$ respectively. Let $m$ denote the order of $st$ in $W$ and let us denote the longest word in $W$ as \be{w_LR} w = s t s \cdots = t s t \cdots \ee where there are $m$ terms in each product above.  Let $\leq$ be the Bruhat order on $W$ and let \be{Sw} S_w= \{ w' \in W \mid w' \leq w \}. \ee By a \emph{string}, we mean a finite sequence of elements $\ty:= (s_1, \ldots, s_r)$ where each $s_i \in \{s, t, 1 \}.$ For each such string, $\ty= (s_1, \ldots, s_m)$ as above, let $\ty_{\star}$ be the corresponding element in the group \be{rk2-gp} \la s, t \ra_\star:= \la s, t \mid s^2=t^2=1\ra; \ee we denote this element by \be{ty*} \ty_\star:= s_{i_1} \star s_{i_2} \cdots  ,\ee which is obtained from $\ty$ by omitting every $s_i=1$ as well as removing every pair of adjoining elements with $s_i=s_{i+1}.$ We also denote by $\ty_o \in W$ the element obtained by multiplying together the elements in $\ty.$

\spoint Let us define two strings  \be{wLR} \begin{array}{lcr} \tw_{L}:= (s, t, s, \cdots) & \text{ and } & \tw_{R}:= (t, s, t, \cdots) \end{array} \ee where both strings have $m$ terms.  Let $\tS_L$ and $\tS_R$ denote the set of strings obtained from $\tw_L$ and $\tw_R$ respectively by replacing some of the $s$ or $t$ by $1.$ Let $\tS_{L, \star}$ (resp. $\tS_{R, \star}$) denote the set of words in $\la s, t \ra_\star$ obtained from $\tS_L$ by sending each $\ty  \mapsto \ty_\star.$ In our previous notation, $\tw_{L, \star} = s \star t \star s \cdots$ and $\tw_{R, \star} = t \star s \star t \cdots.$ Note that as sets of words in $\la s, t \ra_\star,$ \be{eq-sets} \tS_{L, \star} \setminus \{ \tw_{L, \star} \} = \tS_{R, \star} \setminus \{ \tw_{R, \star} \}. \ee As every element in $S_{w} \setminus \{ w \}$ has only one reduced decomposition, we obtain

 \begin{nclaim} \label{SL=SR} The map $\varphi: \la s, t \ra_\star \rr W$ induced by multiplication (previously denoted $\ty_\star \mapsto \ty_o$) induces a bijection $\varphi: \tS_{L, \star} \setminus \{ \tw_{L, \star} \} \rr S_{w} \setminus \{ w \}$ (and so also $\varphi: \tS_{R, \star} \setminus \{ \tw_{R, \star} \} \rr S_{w} \setminus \{ w \}$). \end{nclaim}

\spoint By Lemma \ref{cg-basic-prop}(1), the group $\la s, t \ra_\star$ acts on $\C_{v, S}[\Lv].$ Using this representation, the first statement in Proposition \ref{well-def} is equivalent to showing,
\be{wa-braid} \tw_{L, \star} \star e^{\lv} = \tw_{R, \star} \star e^{\lv}  \text{ for  every } \lv \in \Lv, \ee and the second is equivalent to showing \be{DL-braid} \T_L (e^{\lv}) = \T_R(e^{\lv})  \text{ for every } \lv \in \Lv \ee where, setting $\T_s:= \T_{w_a}$ and  $\T_t:= \T_{w_b},$  \be{TLR} \begin{array}{lcr} \T_L:=\T_{s} \T_{t} \T_{s} \cdots   & \text{ and } & \T_R: = \T_{s} \T_{t} \T_{s} \cdots  \end{array} \ee with $m$ factors on both sides of each of the above equations. From Remark \ref{rmk:DL-braid} we know that \eqref{DL-braid} holds for each $\lv \in \Lv_+.$ We will argue that this implies \eqref{wa-braid} holds for all $\lv \in \Lv_+$ and hence, by Lemma \eqref{weak->strong} that it holds for all $\lv \in \Lv.$ In turn, this will imply \eqref{DL-braid} holds for all $\lv \in \Lv$ as well.

\newcommand{\uT}{\underline{\T}}
\newcommand{\uc}{\underline{\c}}
\newcommand{\ub}{\underline{\b}}
\newcommand{\uth}{\underline{\theta}}
\renewcommand{\f}{\mathsf{f}}
\renewcommand{\g}{\mathsf{g}}

\newcommand{\twist}{\C_v(\Lv_0)[\la s, t \ra_\star]}
\renewcommand{\gv}{\gamma^{\vee}}
\newcommand{\tLv}{\widetilde{\Lambda}^{\vee}}
\newcommand{\tnice}{V}


\spoint  Let $\tLv$ be the lattice spanned by $\m(\av)\av$ and $\m(\bv)\bv.$ Note that $\tLv \subset \Lv_0,$ and that $\b(\gv)$ and $\c(\gv),$ defined as in \eqref{c:b} for $\gv \in R^{\vee},$ lie in the ring $V:= \nicer \cap \nice.$ Define a twisted multiplication on $V[\la s, t \ra_{\star}]$ via \be{twist-V} f [w] \star g [w'] = f g^w [w \star w'] \text{ where } f, g \in V, w, w' \in \la s, t \ra_\star \ee and $g^w$ denotes the \emph{usual} action of $w$ on $g.$


Shifting our point of view slightly, $\T_s$ and $\T_t$ (introduced in \eqref{T:c-b}) may be regarded as elements in the twisted group algebra $V[\la s, t \ra_{\star}]$. For any product $\T:= \T_{s_1} \cdots \T_{s_r}$ we may expand via \eqref{twist-V} to write \be{exp-T} \T := \sum_{\ty \in \la s, t \ra_{\star}} a_{\ty} [\ty]. \ee The application of the right-hand side of this formula to an element of $g \in \C_{v, S}[\Lv]$ (using the natural representation of $\la s, t \ra_{\star}$) agrees with $\T_{s_1} \cdots \T_{s_r}(g)$ using Lemma \ref{cg-basic-prop}(2).  

\spoint Let us now write (expanding as above) \be{T-exp} \begin{array}{lcr} \T_{L} = \sum_{\ty \in S_{L, \star}} f_{\ty} \; \;   [\ty] & \text{ and } &  \T_{R} = \sum_{\ty \in S_{R, \star}} g_{\ty} \; \;   [\ty]
\end{array} \ee where each $f_{\ty}, g_{\ty} \in \C_v(\tLv).$ Using Lemma \ref{cg-basic-prop}(2), we find that the application of the right-hand side of $\sum_{\ty \in S_{L, \star}} f_{\ty} \; \;   [\ty]$ to an element in $\C_{v, S}[\Lv]$ is the same as that of $\T_L$. The same is true for $\sum_{\ty \in S_{R, \star}} g_{\ty} \; \;   [\ty]$ and $\T_{R}.$

\begin{nlem} \label{fg:n} For any $\ty \in S_{L, \star} \setminus \{ \tw_{L, \star} \}= S_{R, \star} \setminus \{ \tw_{R, \star} \}$ we have $f_{\ty} = g_{\ty}.$ Furthermore, we also have $f_{\tw_L}= g_{\tw_R}.$  \end{nlem} 

Assuming this Lemma (which is proved below) and Claim \ref{SL=SR}, we obtain an equality of elements of $V[\la s, t \ra_{\star}]$ \be{LR:1} \T_L - \T_R = f_{\tw_{L}} \tw_{L, \star} - g_{\tw_{R}} \tw_{R, \star} = f_{\tw_L} ( \tw_{L, \star} - \tw_{R, \star} )  \ee Now, from \eqref{LR:1} and the fact that $\T_L(e^{\lv}) = \T_R(e^{\lv})$ for $\lv \in \Lv_+$, we find that \be{weak:conc} \tw_{L, \star} (e^{\lv}) = \tw_{R, \star} (e^{\lv}) \text{  for every } \lv \in \Lv_+. \ee Using Lemma \ref{weak->strong}, we conclude that the previous statement holds for any $\lv \in \Lv$, i.e. we have obtained \eqref{wa-braid}. From \eqref{LR:1}, we now deduce \eqref{DL-braid}.

\tpoint{Proof of Lemma \ref{fg:n}, step 1}  \label{non-met} Let us first recall a non-metaplectic analogue of \eqref{DL-braid}. Let $\C_v(\Lv)[W]$ be the twisted group algebra of $W$ (i.e. $f [w] g [w'] = f g^w [w w']$ where $f, g \in \C_v(\Lv)$ and $w, w' \in W$).  Define the expressions $\uT_a \in \C_v(\Lv)[W]$ by \be{T-nor} \uT_a = \uc(\av) [w_a] + \ub(\av) [1] \ee where for any $a \in R$ we define \be{c:a:nor} \begin{array}{lcr} \uc(\av):= \frac{1 - q^{-1} e^{-\av}}{1 - e^{\av}} & \text{ and } & \ub(\av) := \frac{1 - q^{-1}}{1- e^{\av}} \end{array}. \ee Denote by $\uT_s:= \uT_{a}$ and $\uT_t:= \uT_{b}.$ Let $\uT_L$ and $\uT_R$ be defined as in \eqref{TLR} with each $\T_s, \T_t$ replaced by $\uT_s, \uT_t.$
Expanding, we may write \be{uT:exp} \begin{array}{lcr} \uT_L = \sum_{w' \leq w} \f_{w'} [w']  & \text{ and } \uT_R = \sum_{w' \leq w} \g_{w'} [w'] \end{array} \ee where each $\f_{w'}, \g_{w'}$ is a sum of products of $\uc(\gv)$'s and $\ub(\delta^{\vee})$ with $\gv, \delta^{\vee} \in R^{\vee}.$  From \cite[(8.3)]{lus-K} we find that \be{fg:1} \f_{w'} = \g_{w'} \in \C_{v, S}[\Lv] \text{ for each } w' \leq w. \ee In general the terms $\f_{w'}$ and $\g_{w'}$ are not particularly easy to compute, but one does have \be{fg:long} \f_{w_L} = \g_{w_R} = \prod_{ a > 0, w^{-1} a < 0 } \uc(\av). \ee 

\tpoint{Proof of Lemma \ref{fg:n}, step 2} Recall the bijection $\varphi$ from Claim \ref{SL=SR}. Using \eqref{twist-V} we find that $f_{\ty}$ in \eqref{T-exp} is given by the same sum of products of rational functions as in $\f_{\varphi(\ty)}$ (defined in the non-metaplectic setting \eqref{uT:exp} above) but by replacing $\uc, \ub$ with $\c, \b.$ A similar relation holds between $g_{\ty}$ and $\g_{\varphi(\ty)}.$ Thus, from \eqref{fg:long} we can immediately conclude that $f_{\tw_L} = g_{\tw_R}.$ It remains to see that $f_{\ty} = g_{\ty}$ for any $\ty \in S_{L, \star} \setminus \{ \tw_{L, \star} \}= S_{R, \star} \setminus \{ \tw_{R, \star} \}.$ 

%

\tpoint{Proof of Lemma \ref{fg:n}, step 3} We recall some facts about rank two root systems. Let $\B$ and $\Q$ be as in \eqref{B:Q}. If $\B(\av,\bv)= 0,$ (i.e. $R^{\vee} =A_1\times A_1$), then  $w_aw_b$ has order $m=2,$ and $f_{\ty}$ and $g_{\ty}$ are clearly identical for every $\ty \in \la s, t \ra_{\star}$. So we may omit this case in the discussion that follows. The next claim is verified empirically.

\begin{nclaim} Assume $R^{\vee}$ is a rank two root system with $\Pi^{\vee}=\{ \gv_1, \gv_2 \}.$ Assume $\gv_1$ is the short root, and that $\B(\gv_1, \gv_2)\neq 0.$ Then one has the following possibilities
\begin{enumerate}
\item[(i)] $\Q(\gv _2)/\Q(\gv _1)=1,$ $\B(\gv_1,\gv_2)=(-1)\cdot \Q(\gv_1),$ and $R^{\vee}$ is of type $A_2.$
\item[(ii)] $\Q(\gv _2)/\Q(\gv _1)=2,$ $\B(\gv_1,\gv_2)=(-2)\cdot \Q(\gv_1),$ and $R^{\vee}$ is of type $B_2.$
\item[(iii)] If $\Q(\gv _2)/\Q(\gv _1)=3,$ $\B(\gv_1,\gv_2)=(-3)\cdot \Q(\gv_1),$ and $R^{\vee}$ is of type $G_2.$
\end{enumerate}
\end{nclaim}

\noindent Let $R^{\vee}=\{\av, \bv \}$ be our rank two root system (with $\av$ the short root) which we assume is not of type $A_1 \times A_1.$ Let $\Lv$ and $\tLv$ be the lattices spanned by $\av, \bv$ and $\m(\av)\av, \m(\bv)\bv$ respectively.

\begin{nprop} \label{prop-rk2}
There is an isomorphism $\psi: \tLv \cong \Lv$ determined by either \be{isom1}\begin{array}{lcr}  \m(\av)\av \mapsto \av,\,\  \m(\bv)\bv\mapsto \bv & \text{ or } &   
\m(\bv)\bv\mapsto  \av,\,\  \m(\av)\av\mapsto \bv. \end{array} \ee
\end{nprop}
\begin{proof} As $\Lv, \tLv$ are two lattices in the same Euclidean space whose generators are proportional to one another, to verify they are isomorphic one just needs to consider the ratio of their lengths. Equivalently, we can consider the ratio of the function $\Q$ on the basis vectors.

If $\m(\av)=\m(\bv),$ the first map above gives an isomorphism, so assume this is not the case. In the trichotomy of the Claim above, we are then in case (ii) or (iii). Using the definition \eqref{ma:def} and the fact that $\Q(\bv)/\Q(\av) \in \{ 2, 3 \}$, we can easily see: if $\m(\av) \neq \m(\bv),$ then $\m(\av)/\m(\bv) = \Q(\bv)/ \Q(\av).$ Hence \be{QB_quot}\frac{\Q(\m(\bv)\bv)}{\Q(\m(\av)\av)}=\frac{\Q(\bv)}{\Q(\av)}\cdot \frac{\m(\bv)^2}{\m(\av)^2} = \frac{\Q(\av)}{\Q(\bv)},\ee and so the second map in \eqref{isom1} results in an isomorphism.    \end{proof}

\tpoint{Proof of Lemma \ref{fg:n}, step 4} Finally we finish the proof of the Lemma. Let $\sigma: \la s, t \ra_{\star} \rr \la s, t \ra_{\star}$ be the map which interchanges $s$ and $t.$ The previous Proposition \ref{prop-rk2} induces an isomorphism which we also denote by $\psi: V \rr \C_{v, S}[\Lv] .$ Together with the bijection $\varphi$ from Claim \ref{SL=SR}, or with $\varphi\circ \sigma ,$ $\psi$ extends to a map of vector spaces $V[\la s, t \ra_{\star}]\rightarrow \C_{v, S}[\Lv][W]$ mapping $\{\T_s,\T_t\}$ to $\{\uT_a,\uT_b\}$ which is injective on the subspace $V[ \la s, t \ra_{\star}\setminus\{\tw_{L,\star},\tw_{R,\star}\}].$ Under this map $f_{\ty} \mapsto \f_{\varphi(\ty)}$ or $f_{\ty} \mapsto \g_{\varphi(\sigma(\ty))}.$ In the former case, we also have $g_{\ty} \mapsto \g_{\varphi(\ty)},$ and in the latter case $g_{\ty} \mapsto \f_{\varphi(\sigma(\ty))}.$ The Lemma now follows from \eqref{fg:1}.


\begin{bibsection}
\begin{biblist}

\bib{bgkp}{article}{
   author={Braverman, A.},
   author={Garland, H.},
   author={Kazhdan, D.},
   author={Patnaik, M.},
   title={An affine Gindikin-Karpelevich formula},
   conference={
      title={Perspectives in representation theory},
   },
   book={
      series={Contemp. Math.},
      volume={610},
      publisher={Amer. Math. Soc., Providence, RI},
   },
   date={2014},
   pages={43--64},
   review={\MR{3220625}},
   doi={10.1090/conm/610/12193},
}
	


\bib{bkp}{article}{
   author={Braverman, A.},
   author={Kazhdan, D.},
   author={Patnaik, M. M.},
   title={Iwahori-Hecke algebras for $p$-adic loop groups},
   journal={Invent. Math.},
   volume={204},
   date={2016},
   number={2},
   pages={347--442},
   issn={0020-9910},
   review={\MR{3489701}},
   doi={10.1007/s00222-015-0612-x},
}

\bib{bbf-annals}{article}{
    author={Brubaker, B.},
    author={Bump, D.},
    author={Friedberg, S.},   
     TITLE = {Weyl group multiple {D}irichlet series, {E}isenstein series
              and crystal bases},
   JOURNAL = {Ann. of Math. (2)},
  FJOURNAL = {Annals of Mathematics. Second Series},
    VOLUME = {173},
      YEAR = {2011},
    NUMBER = {2},
     PAGES = {1081--1120},
      ISSN = {0003-486X},
     CODEN = {ANMAAH},
   MRCLASS = {11F03 (17B37)},
  MRNUMBER = {2776371},
       DOI = {10.4007/annals.2011.173.2.13},
       URL = {http://dx.doi.org/10.4007/annals.2011.173.2.13},
}

\bib{bbl}{article}{
  title={Whittaker functions and Demazure operators},
  author={Brubaker, B.},
  author={Bump, D.},
  author={Licata, A.},
  journal={Journal of Number Theory},
  volume={146},
  pages={41--68},
  year={2015},
  publisher={Elsevier}
}


\bib{del-bry}{article}{
   author={Brylinski, J.-L.},
   author={Deligne, P.},
   title={Central extensions of reductive groups by $\bold K_2$},
   journal={Publ. Math. Inst. Hautes \'Etudes Sci.},
   number={94},
   date={2001},
   pages={5--85},
   issn={0073-8301},
   review={\MR{1896177 (2004a:20049)}},
   doi={10.1007/s10240-001-8192-2},
}

\bib{ch:gu}{article}{
   author={Chinta, G.},
   author={Gunnells, P. E.},
   title={Constructing Weyl group multiple Dirichlet series},
   journal={J. Amer. Math. Soc.},
   volume={23},
   date={2010},
   number={1},
   pages={189--215},
   issn={0894-0347},
   review={\MR{2552251 (2011g:11100)}},
   doi={10.1090/S0894-0347-09-00641-9},
}

\bib{cgp}{article}{
   author={Chinta, G.},
   author={Gunnells, P. E.},
   author={Pusk\'{a}s , A.}
   title={Metaplectic Demazure operators and Whittaker functions},
   journal={ arXiv:1408.5394}
   }



\bib{ch:of}{article}{
   author={Chinta, G.},
   author={Offen, O.},
   title={A metaplectic Casselman-Shalika formula for ${\rm GL}_r$},
   journal={Amer. J. Math.},
   volume={135},
   date={2013},
   number={2},
   pages={403--441},
   issn={0002-9327},
   review={\MR{3038716}},
   doi={10.1353/ajm.2013.0013},
}


\bib{gg}{article}{
	author={Gan, W. T.},
	author={Gao, F.}
	title={The Langlands-Weissman Program for Brylinski-Deligne extensions}
	journal={arXiv:1409.4039}
	}

\bib{ka:pat}{article}{
   author={Kazhdan, D. A.},
   author={Patterson, S. J.},
   title={Metaplectic forms},
   journal={Inst. Hautes \'Etudes Sci. Publ. Math.},
   number={59},
   date={1984},
   pages={35--142},
   issn={0073-8301},
   review={\MR{743816 (85g:22033)}},
}

\bib{lus-K}{article}{
   author={Lusztig, G.},
   title={Equivariant $K$-theory and representations of Hecke algebras},
   journal={Proc. Amer. Math. Soc.},
   volume={94},
   date={1985},
   number={2},
   pages={337--342},
   issn={0002-9939},
   review={\MR{784189 (88f:22054a)}},
   doi={10.2307/2045401},
}

\bib{mac:aff}{book}{
   author={Macdonald, I. G.},
   title={Affine Hecke algebras and orthogonal polynomials},
   series={Cambridge Tracts in Mathematics},
   volume={157},
   publisher={Cambridge University Press, Cambridge},
   date={2003},
   pages={x+175},
   isbn={0-521-82472-9},
   review={\MR{1976581 (2005b:33021)}},
   doi={10.1017/CBO9780511542824},
}

\bib{mac:for}{article}{
   author={Macdonald, I. G.},
   title={A formal identity for affine root systems},
   conference={
      title={Lie groups and symmetric spaces},
   },
   book={
      series={Amer. Math. Soc. Transl. Ser. 2},
      volume={210},
      publisher={Amer. Math. Soc., Providence, RI},
   },
   date={2003},
   pages={195--211},
   review={\MR{2018362 (2005c:33012)}},
}


\bib{mac:poin}{article}{
   author={Macdonald, I. G.},
   title={The Poincar\'e series of a Coxeter group},
   journal={Math. Ann.},
   volume={199},
   date={1972},
   pages={161--174},
   issn={0025-5831},
   review={\MR{0322069 (48 \#433)}},
}


\bib{mat}{article}{
   author={Matsumoto, H.},
   title={Sur les sous-groupes arithm\'{e}tiques des groupes semi-simples
   d\'eploy\'es},
   language={French},
   journal={Ann. Sci. \'{E}cole Norm. Sup. (4)},
   volume={2},
   date={1969},
   pages={1--62},
   issn={0012-9593},
   review={\MR{0240214 (39 \#1566)}},
}


\bib{mac-cryst}{article}{
  title={Metaplectic Whittaker functions and crystal bases},
  author={McNamara, P. J.},
  journal={Duke Mathematical Journal},
  volume={156},
  number={1},
  pages={1--31},
  year={2011},
  publisher={Duke University Press}
}

\bib{mac-prin}{article}{
   author={McNamara, P. J.},
   title={Principal series representations of metaplectic groups over local fields},
   conference={ title={Multiple Dirichlet series, L-functions and automorphic forms}, },
   book={
      series={Progr. Math.},
      volume={300},
      publisher={Birkh\"auser/Springer, New York},
   },
   date={2012},
   pages={299--327}
}
		

\bib{mac-met}{article}{
   author={McNamara, P. J.},
   title={The metaplectic Casselman-Shalika formula},
   journal={Trans. Amer. Math. Soc.},
   volume={368},
   date={2016},
   number={4},
   pages={2913--2937},
   issn={0002-9947},
   review={\MR{3449262}},
   doi={10.1090/tran/6597},
}



\bib{pat-whit}{article}{
   author={Patnaik, M. M.}
   title={Unramified Whittaker Functions on $p$-adic Loop Groups},
   journal={American Journal of Mathematics}
   volume={139},
   date={2017},
   number={1},
   pages={175--213},
   doi={10.1353/ajm.2017.0003},
   }

%


\bib{pus-th}{article}{
  title={Whittaker functions on metaplectic covers of GL(r)},
  author={Pusk{\'a}s, A.},
  journal={ arXiv:1605.05400},
  year={2016}
}

\bib{ram}{book}{
   author={Ramakrishnan, D.},
   author={Valenza, R. J.},
   title={Fourier analysis on number fields},
   series={Graduate Texts in Mathematics},
   volume={186},
   publisher={Springer-Verlag, New York},
   date={1999},
   pages={xxii+350},
   isbn={0-387-98436-4},
   review={\MR{1680912 (2000d:11002)}},
   doi={10.1007/978-1-4757-3085-2},
}

\bib{savin}{article}{
   author={Savin, G.},
   title={On unramified representations of covering groups},
   journal={J. Reine Angew. Math.},
   volume={566},
   date={2004},
   pages={111--134},
   issn={0075-4102},
   review={\MR{2039325 (2005a:22014)}},
   doi={10.1515/crll.2004.001},
}

\bib{ser:lf}{book}{
   author={Serre, J.-P.},
   title={Corps locaux},
   language={French},
   series={Publications de l'Institut de Math\'{e}matique de l'Universit\'{e} de
   Nancago, VIII},
   publisher={Actualit\'es Sci. Indust., No. 1296. Hermann, Paris},
   date={1962},
   pages={243},
   review={\MR{0150130 (27 \#133)}},
}


\bib{steinberg}{book}{
   author={Steinberg, R.},
   title={Lectures on Chevalley groups},
   note={Notes prepared by John Faulkner and Robert Wilson},
   publisher={Yale University, New Haven, Conn.},
   date={1968},
   pages={iii+277},
   review={\MR{0466335 (57 \#6215)}},
}
	
\bib{tok}{article}{
   author={Tokuyama, T.},
   title={A generating function of strict Gel\cprime fand patterns and some
   formulas on characters of general linear groups},
   journal={J. Math. Soc. Japan},
   volume={40},
   date={1988},
   number={4},
   pages={671--685},
   issn={0025-5645},
   review={\MR{959093 (89m:22019)}},
   doi={10.2969/jmsj/04040671},
}

\bib{weis}{article}{
   author={Weissman, M. H.},
   title={Split metaplectic groups and their L-groups},
   journal={J. Reine Angew. Math.},
   volume={696},
   date={2014},
   pages={89--141},
   issn={0075-4102},
   review={\MR{3276164}},
   doi={10.1515/crelle-2012-0111},
}


\end{biblist}
\end{bibsection}
\end{document}